\documentclass[11pt]{amsart}
\usepackage{amssymb, graphicx, color}
\usepackage{amsfonts}
\usepackage{amsthm}
\usepackage{enumerate}
\usepackage[mathscr]{eucal}
\usepackage{graphicx}
\usepackage{dsfont}
\usepackage[all]{xy}
\usepackage{comment}

\usepackage[pdftex, bookmarksnumbered, bookmarksopen, colorlinks, citecolor=blue, linkcolor=blue]{hyperref}

\allowdisplaybreaks

\newcommand{\8}{\infty}

\theoremstyle{plain}

\newtheorem{thm}{Theorem}[section]
\newtheorem{defi}[thm]{Definition}

\newtheorem{lem}[thm]{Lemma}

\newtheorem{prop}[thm]{Proposition}

\newtheorem{remark}[thm]{Remark}

\theoremstyle{definition}
\theoremstyle{remark}
\numberwithin{equation}{section}

\newcommand{\R}{{\mathbb R}}
\newcommand{\N}{{\mathbb N}}
\newcommand{\Z}{{\mathbb Z}}

\newcommand{\M}{{\mathcal M}}

\newcommand{\bs}{\begin{split}}
\newcommand{\es}{\end{split}}

\newcommand{\be}{\begin{eqnarray*}}
\newcommand{\ee}{\end{eqnarray*}}
\newcommand{\beq}{\begin{align}}
\newcommand{\eeq}{\end{align}}

\def\Q{\mathcal{Q}}
\def\1{\mathbf{1}}
\def\I{\mathcal{F}}

\begin{document}

%
%
%
%
%
%
%
%
\setcounter{page}{1}
\title[Quantitative mean ergodic inequalities]
{Quantitative mean ergodic inequalities: power bounded operators acting on one single noncommutative $L_p$ space}

\author[G. Hong]{Guixiang Hong}
\address{
Institute for Advanced Study in Mathematics\\
 Harbin Institute of Technology\\
  Harbin 150001\\
  China}

\email{guixiang.hong@whu.edu.cn}
\author[W. Liu]{Wei Liu}
\address{
 School of Mathematical Sciences\\
 Fudan University\\
 Shanghai 200433\\
China}

\email{wliu\_math@fudan.edu.cn}
\author[B. Xu]{Bang Xu}
\address{Department of Mathematics\\
	University of Houston\\
	Houston TX 77204\\
	USA;
	Department of Mathematical Sciences\\
	Seoul National University\\
	Seoul 08826\\
	Republic of Korea}

\email{bangxu@whu.edu.cn}

\thanks{This work
 was partially supported by Natural Science Foundation of China (Grant: 12071355)}


\date{\today.
}
\begin{abstract}
In this paper, we establish the quantitative mean ergodic theorems for two subclasses of power bounded operators on a fixed noncommutative $L_p$-space with $1<p<\8$, which mainly concerns power bounded invertible operators and Lamperti contractions.  Our approach to the quantitative ergodic theorems relies on noncommutative square function inequalities. The establishment of the latter involves several new ingredients such as the almost orthogonality and Calder\'on-Zygmund arguments for non-smooth kernels from semi-commutative harmonic analysis, the extension properties of the operators under consideration from operator theory, and a noncommutative version of the classical transference method due to Coifman and Weiss. 




\end{abstract}

\maketitle

\section{Introduction}\label{introdu}
In classical ergodic theory, there are many works related to the convergence properties of certain averages along the orbits with respect to the transformations. Let $(X,\mathcal F,\mu)$ be a $\sigma$-finite measure space. The celebrated von Neumann mean ergodic theorem \cite{vonNeu32} states that when $T$ is a unitary operator on $L_2(X,\mu)$ induced by a $\mu$-preserving measurable transformation on $X$, the ergodic averages $M_{n}f$ defined by
\begin{equation}\label{ergodic-ave}
  M_{n}f(x)=\frac{1}{n+1}\sum_{k=0}^{n}T^{k}f(x)\ \ n\in\N,
\end{equation}
converges in $L_2(X)$ for any $f\in L_2(X)$. Later on, Riesz~\cite{Rie38} greatly generalized the von Neumann mean ergodic theorem; {he proved that the convergence of ergodic averages is also valid for contractive operators defined  simultaneously on all $L_{p}(X,\mu)$ ($1\leq p<\infty$) spaces, where $(X,\mu)$ is a probability space.} Also, Riesz~\cite{Rie41} gave a simple proof when $T$ is a contraction operator on some Hilbert space.  Furthermore, the mean ergodic theorem for contraction operators acting on general $L_p(X,\mu)$ spaces for all $1\le p\le\8$ was established by Dunford and Schwartz ~\cite[VIII]{DS58}.

It is then natural to ask for the speed of the convergence of the ergodic averages. Unfortunately,  Krengel~\cite{Kre} proved that the speed of the ergodic convergence can be arbitrarily slow. On the other hand, one can not capture any information on the rate of the convergence from the classical proofs of aforementioned works.  With the aid of the spectral theorem and the dilation theorem developed in~\cite{Sz-Fo70}, Jones, Ostrovskii and Rosenblatt~\cite{JOR96} established the square function inequalities for ergodic averages~\eqref{ergodic-ave} associated with a contraction on $L_2(X)$. More precisely, they proved that for a $L_2$-contraction $T$ and any sequence of finite positive integers $n_0<n_1<\cdots<n_m$,
\begin{align}\label{norm-vari}
  \Big(\sum_{i=1}^m\big\|M_{n_i}(T)f-M_{n_{i-1}}(T)f\big\|^2_{L_2(X)}\Big)^{1/2}\le 25\|f\|_{L_2(X)}.
\end{align}
This result can  be viewed as a quantitative and finer version of the mean ergodic theorem. Indeed, \eqref{norm-vari} implies that for any $\varepsilon>0$, the sequence $(M_n(f))_{n=1}^\infty$ admits at most $625(\varepsilon^{-2}\|f\|_2^2)$ jumps of size at least $\varepsilon$ in $L_2$ norm, and as a consequence the sequence converges. Moreover, many variants of the inequality~\eqref{norm-vari} have been obtained. Avigad and Rute~\cite{AR} extended~\eqref{norm-vari} with the power 2 of~\eqref{norm-vari} replaced by $q$ ($q\ge2$) to $q$-uniformly convex Banach spaces and specific power bounded operators. Bourgain~\cite{Bour89} considered the corresponding variational inequalities which deduces the pointwise convergence of ergodic averages immediately  and also the quantitative ergodic inequality (\ref{norm-vari}). For the latter, we refer the reader to \cite{JRW98, JRW03} due to Jones {\it et al} for more details. We remark that Calder{\'o}n's transference principle plays an important role in the above works which reduces~\eqref{norm-vari} to the study of the related operators in harmonic analysis where more tools are available.





Motivated by quantum physics, the convergence of the ergodic averages in von Neumann algebras has attracted much attention. For instance, Kov\'{a}cs and Sz\"{u}cs \cite{Ko-Sz} considered the mean ergodic theorem for an automorphism $T$ on von Neumann algebras equipped with a faithful $T$-invariant semifinite normal state, and Lance~\cite{Lan76} gave a complete discussion of this subject. Later on, Yeadon \cite{Ye, Yean80} established the mean ergodic theorem for positive Dunford-Schwartz operators on noncommutative $L_p$ spaces. For more results on the mean ergodic theorem in von Neumann algebras we refer the reader to~\cite{Ja1,Ja2}. On the other hand, on the study of pointwise ergodic theorems in noncommutative $L_p$, after the work in the case $p=\infty$ \cite{Lan76, Kum78, CoDa78} and in the case $p=1$ \cite{Ye}, a breakthrough was made by Junge and Xu~\cite{JX}, where they established the noncommutative maximal ergodic inequalities for positive Dunford-Schwartz operators. This celebrated work  motivated further research on noncommutative ergodic theory, such as \cite{An,Be,HS,Hu,Li}. In particular, the first author and his collaborators broke the framework of Junge and Xu by establishing the maximal and pointwise ergodic theorems for a large subclass of positive operators on one single noncommutative $L_p$ space, see~\cite{HLW, HRW} for more details.

However, to the best of the authors' knowledge, there is no quantitative estimate of noncommutative ergodic theorems in the literature.
This paper is devoted to the first study of quantitative mean ergodic theorem (\ref{norm-vari}) under the noncommutative framework. To better state our results, we need to introduce some notions. Let $\M$ be a semifinite von Neumann algebra equipped with a normal semifinite faithful (abbreviated as \emph{n.s.f}) trace $\tau$. Let $L_p(\M)$ be the associated noncommutative $L_{p}$ space and $L_p(\M;\ell^{rc}_{2})$ be one noncommutative analogue of Hilbert-valued $L_{p}$ space (see Section~\ref{Pre-Non} for the precise definition). Throughout the paper, $T$ stands for a bounded linear operator on $L_p(\mathcal{M})$.
The one-sided ergodic averages $M_n(T)$ is defined as
$$M_{n}(T)=\frac{1}{n+1}\sum_{k=0}^{n}T^{k},\quad \ \forall\ n\in\N;$$
and if $T$ is invertible, one may define the two-sided ergodic averages $B_{n}(T)$ by
$$B_{n}(T)=\frac{1}{2n+1}\sum_{k=-n}^{n}T^{k},\ \ \forall\ n\in\N.$$

\begin{thm}\label{theorem11}
Let $1<p<\infty$. Suppose that $T$ be an invertible power bounded operator,
\begin{equation}\label{con-1}
  \sup_{k\in\Z}\|T^{k}:L_p(\M)\rightarrow L_p(\M)\|<\infty.
\end{equation}
Then there exists a positive constant $C_{p}$ such that
%
$$\sup\big\|\big(M_{n_{i}}(T)x-M_{n_{i+1}}(T)x\big)_{i\in \N}\big\|_{L_p(\M;\ell^{rc}_{2})}\leq C_{p}\|x\|_{L_{p}(\M)},\forall x\in L_p(\M)$$
where the supremum is taken over all the increasing subsequence $(n_{i})_{i\in\N}$ of positive integers. Similar inequality holds for two-sided ergodic averages $(B_n(T))_{n\in\mathbb N}$.
\end{thm}

\begin{remark}{\rm
Theorem~\ref{theorem11} is a noncommutative version of~\cite[Theorem 1.2]{JOR96}. Note also that Theorem~\ref{theorem11} can be viewed as an ergodic theorem with respect to a bounded $L_p(\mathcal M)$-representation of the group $\mathbb Z$. This motivates us to further consider
 the bounded noncommuative $L_p$-representations of other groups such as the ones of polynomial growth that appeared in \cite{HLW}. Similar quantitative mean ergodic theorems will appear in the subsequent paper \cite{HLX}.}
\end{remark}

The second main result concerns the quantitative mean ergodic theorem for Lamperti operators.
\begin{defi}\rm
 An operator $T$ is called a Lamperti operator (or supports separating) if for any two $\tau$-finite projections $e,f\in\mathcal{M}$ satisfies $ef=0$, then
$$(Te)^*Tf=Te(Tf)^*=0.$$
\end{defi}

In~\cite{HRW},  the authors established the maximal inequalities for the convex combinations of positive Lamperti contractions on one single noncommutative $L_p$ spaces, which can be viewed as the first Akcoglu's maximal ergodic inequalities in the noncommutative setting.

In this paper, we establish the quantitative mean ergodic theorem for Lamperti operators where the positivity assumption can be relaxed. 


\begin{thm}\label{thm-1}
	Let $1< p<\infty$. Suppose that $T$ belong to the family
\begin{equation}\label{con-lam}
	\mathfrak{S}=\overline{\mbox{conv}}^{\mbox{sot}}\{S:L_p(\mathcal M)\rightarrow L_p(\mathcal M)~\mbox{Lamperti contractions}\},
\end{equation}	
that is, the closed convex hull of all Lamperti contractions on $L_p(\mathcal M)$ with respect to strong operator topology.	For $2\leq p<\infty$, there exists a positive constant $C_{p}$ such that for any increasing subsequence of positive integers $(n_{i})_{i\in\N}$,
$$\big\|\big(M_{n_{i}}(T)x-M_{n_{i+1}}(T)x\big)_{i\in \N}\big\|_{L_p(\M;\ell^{rc}_{2})}\leq C_{p}\|x\|_{L_{p}(\M)}\;\forall x\in L_p(\M).$$
For $1<p<2$, if $T\in \mathfrak{S}$ is positive, then the above conclusion holds too.

\end{thm}

\bigskip

\begin{remark}{\rm
(i) Theorem \ref{theorem11} and Theorem \ref{thm-1} 
seem new even in the commutative case since the quantitative mean ergodic inequalities are deduced for a large class of operators acting on a fixed  $L_p$ spaces with $1<p\neq2<\8$. It should be pointed out that when $\mathcal M$ is commutative, the class of operators in the two theorems should be able to be enlarged. We left this to the interested reader.

 (ii) In the remarkable paper \cite{AR}, the authors obtained a form of quantitative mean ergodic theorems for power-bounded operators on a general Banach space $X$ having martingalge cotype $p$ property with $p\geq2$,
\begin{align}\label{AR}\sum_i\|M_{n_{i}}(T)x-M_{n_{i+1}}(T)x\|^p_{X}\leq C_{p}\|x\|^p_{X}\;\forall x\in X.\end{align}
   More interestingly, they exploited this quantitative estimate to provide lower bounds on the rate of metastability, which is related to Kohlenbach's `proof mining' program (see the end of \cite{AR} and the references therein).  It is easy to observe that our Theorems \ref{theorem11} and \ref{thm-1} strengthen greatly \eqref{AR} in the case $X$ being the noncommutative $L_p$ spaces, and thus one might  expect further applications. We will take care of this elsewhere.}
 \end{remark}

With a moment's thought, there are many difficulties to apply the classical methods exploited in~\cite{A, AR, JOR96, JRW98} to our setting. Our approach is mainly motivated by the study of maximal ergodic inequalities and semi-commutative harmonic analysis.


The proof of Theorem \ref{theorem11} and Theorem \ref{thm-1} rely on several auxiliary results. Our first key ingredient is to explore the noncommutative version of the classical transference principle due to Coifman and Weiss \cite{CoWe}. Before doing this, we introduce some notation. Let $(n_{i})_{i\in\N}$ be an increasing sequence in $\N$. A sequence of intervals $(A_{n_i})_{i\in\N}\subset\Z$ is called {\it nested} if it satisfies one of the following two cases:
\begin{enumerate}[\noindent]
\item~{(a)}~each $A_{n_i}$  can be written as $[0,n_i]$;
\item~{(b)}~each $A_{n_i}$  can be written as $[-n_i,n_i]$.
\end{enumerate}
Let $A\subset\Z$ be an interval and $f: \Z\rightarrow S_{\M}$ be a bounded operator-valued function, where $S_{\M}$ is the subset of $\M$ with $\tau$-finite support. The averaging operator over $A$ is defined by
\begin{equation*}
  M_A f(v)= \frac{1}{|A|} \sum_{y\in A}f(v+y), \quad v\in\Z.
\end{equation*}
Given a nested sequence $(A_{n_i})_i$, set $T_i=M_{A_{n_i}}-M_{A_{n_{i+1}}}$.
\medskip

\begin{prop}\label{trans}
 Let $1<p<\8$ and $(A_{n_i})_i$ be the case {\rm (a)}.
 If there exists a positive constant $C_p$ such that
\begin{equation}\label{assume-1}
      \| (T_{i}f)_{i\in\N}\|_{L_p(\mathcal{N};
\ell_{2}^{rc})}\leq C_{p} \|f\|_{L_p(\mathcal N)}\ \forall f\in L_p(\mathcal N),
    \end{equation}
then for any invertible power bounded operator $T$, there exists an absolute constant $C$ such that
$$\big\|\big(M_{n_{i}}(T)x-M_{n_{i+1}}(T)x\big)_{i\in \N}\big\|_{L_p(\M;\ell^{rc}_{2})}\leq CC_{p}\|x\|_{L_{p}(\M)}\ \forall x\in L_p(\M).$$

A similar assertion holds for two-sided ergodic averages $(B_{n_i}(T))_{i}$.
\end{prop}

The above noncommutative transference technique is partly motivated by noncommutative Calder\'on's transference principle developed in \cite{HLW}. According to the transference principle, to obtain~\eqref{assume-1}, we need the operator-valued square function inequalities related to the Hardy-Littlewood averages on $\Z$.

\smallskip

\begin{thm}\label{Th5}
Let $1\le p\le \8$ and $(n_{i})_{i\in\N}$ be an increasing sequence of positive integers. Then the following assertions are true with a
positive constant $C_p$ depending only on $p$:
\begin{enumerate}[\noindent]
\item\emph{(i)}~for $p=1$,
$$\|(T_{i}f)_{i\in\N}\|_{L_{1,\infty}(\mathcal{N},
\ell_{2}^{rc})}\leq C_p\|f\|_{1},\; \forall f\in L_{1}(\mathcal N);$$

\item\emph{(ii)}~for $p=\infty$,
$$\Big\|
\sum_{i:i\in\N} T_i \hskip-1pt f \otimes e_{1i}
\Big\|_{\mathrm{BMO}_{d}(\mathcal{R})} +
\Big\|
\sum_{i:i\in\N} T_i \hskip-1pt f \otimes e_{i1}
\Big\|_{\mathrm{BMO}_{d}(\mathcal{R})} \leq C_p\,
\|f\|_\infty,\; \forall f\in L_{\infty}(\mathcal N);$$
\item\emph{(iii)}~for $1<p<\8$,
$$\| (T_{i}f)_{i\in\N}\|_{L_p(\mathcal{N};
\ell_{2}^{rc})}\leq C_{p} \|f\|_{p}, \; \forall f\in L_{p}(\mathcal N).$$
\end{enumerate}
Here $\mathcal{N}=L_\infty(\Z)\overline{\otimes}\M$ is equipped with the tensor trace $\varphi=\sum_{\mathbb Z}\otimes \tau$ and $\mathcal{R}=\mathcal{N}\overline{\otimes}\mathcal{B}(\ell_2)$ with the tensor trace $\varphi\otimes tr$, where $tr$ is the canonical trace on $\mathcal{B}(\ell_2)$.
\end{thm}

{Theorem \ref{Th5} is an operator-valued version of the square function inequalities that were obtained in the fundamental works \cite{JOR96,JRW98} for  $p\le 2$, and in \cite{Hong} (see also \cite{HM1}) for $p>2$. We refer to Section \ref{strong} for the definition of the dyadic $\mathrm{BMO}$ space $\mathrm{BMO}_{d}(\mathcal R)$. 

\medskip

To establish Proposition \ref{trans}, we need the following extension property of the bounded linear operators.

\begin{lem}\label{exten-lin}
Let $1\le p<\8$. Assume that $T$ is a bounded linear operator on $L_p(\mathcal M)$. Then $T$ extends to a bounded operator on $L_p(\mathcal M;\ell^{rc}_2)$.
\end{lem}

Lemma \ref{exten-lin} follows easily from the noncommutative Khintchine inequalities, Proposition \ref{nonkin}, which should be known to experts. So we leave the proof to the interested reader. {Then it is easy to see that Theorem \ref{theorem11} is a product of Theorem \ref{Th5} and Proposition~\ref{trans}.}

\medskip

However, in sharp contrast to showing Theorem \ref{theorem11}, to exploit the transference technique in order to prove Theorem \ref{thm-1} is much more complicated. To this end, we start with a reduction. More precisely, we first reduce Lamperti operators to isometries by applying the structural characterizations and dilation properties of Lamperti operators recently developed in~\cite{HRW} (see Section~\ref{main-thm2}). With this reduction, the effort will be devoted to the strong type $(p,p)$ estimate of the square function for isometric operators. This will be achieved by using a similar noncommutative  transference principle as Proposition \ref{trans} for isometries, {\it once} there holds the following property that concerns the isometric extension of isometries to $L_p(\mathcal M;\ell^{rc}_2)$.

\begin{prop}\label{extend of Lam}
	Let $1< p<\infty$ and $T:L_p(\mathcal M)\rightarrow L_p(\mathcal M)$ be an isometry. Then $T$ extends to an isometry (resp. a contraction) on $L_p(\mathcal M;\ell_2^{rc})$ if $2\le p<\8$ (resp. $1< p<2$). If $T$ is morevoer positive, then $T$ extends also to an isometry on $L_p(\mathcal M;\ell_2^{rc})$ for $1<p<2$.
\end{prop}
When $\mathcal M$ is commutative, the above extension of Proposition~\ref{extend of Lam} is almost plain. The truly noncommutative case is highly non-trivial. Our argument depends on the structural description of isometries, see e.g. \cite{Jun-Ruan-Sher05, Yea81}. Moreover, for $1<p<2$, some complicated duality argument are explored, and similar one has appeared in \cite{HRW} in dealing with noncommutative maximal inequalities. For more details we refer to Section~\ref{main}.

\begin{remark}\label{app}{\rm
What we need to point out here is that the  estimates stated in all the aforementioned theorems for infinite summations should be understood as a consequence of the corresponding uniform boundedness for all finite summations by the standard approximation arguments (see e.g. \cite[Section 6.A]{JMX}). For this reason, as in \cite{HX}, we are not going to explain the convergence of infinite sums appearing in the whole paper if there is no ambiguity.}
\end{remark}

We end the introduction by mentioning the organization of the paper. In Section \ref{Pre-Non}, we recall the necessary background including noncommutative $L_p$ spaces and Hilbert-valued $L_p$ spaces, as well as the noncommutative Calder{\'o}n-Zygmund decomposition recently developed in~\cite{CCP}. Section~\ref{redu}-\ref{strong} is devoted to the proof of Theorem~\ref{Th5}. In Section~\ref{main}, we prove  Proposition \ref{trans} and Theorem~\ref{theorem11}. Finally, in Section~\ref{main-thm2}, we  prove Proposition~\ref{extend of Lam} and Theorem~\ref{thm-1}, which involves the intermediate square function inequalities for isometries---Lemma~\ref{thm-iso}.

\smallskip


\textbf{Notation:} In all what follows, we use the same letter $C$ to denote various positive constants
that may change at each occurrence. Also, we write
$X\lesssim Y$ for non-negative quantities $X$ and $Y$ to mean that $X\le CY$ for some inessential constant $C>0$. Similarly, we use the notation $X\thickapprox Y$ if both $X\lesssim Y$ and $Y \lesssim X$ hold.

\section{Preliminaries}\label{Pre-Non}
\subsection{Noncommutative $L_{p}$ spaces}\quad

\medskip

Throughout this paper, $\M$ denotes a semifinite von Neumann algebra equipped with a \emph{n.s.f} trace $\tau$. Let $\M_{+}$ be the cone of positive elements in $\M$. Given $x\in\M_{+}$, the support projection of $x$, denoted by $\mathrm{supp}x$, is defined as the least projection $e$ in $\M$ such that $ex=xe=x$. Let $\mathcal{S_{\M+}}$ be the set of all $x\in\M_{+}$ such that $\tau(\mathrm{supp}x)<\infty$ and  $\mathcal{S}_{\M}$ be the linear span of $\mathcal{S_{\M+}}$. Then $\mathcal{S}_{\M}$ is a $w^{*}$-dense $\ast$-subalgebra of $\M$. Given $0< p<\infty$ and $x\in\mathcal{S}_{\M}$, if we set
$$\|x\|_{p}=\big(\tau(|x|^p)\big)^{1/p},$$
where $|x|=(x^{\ast}x)^{\frac{1}{2}}$ is the modulus of $x$, then it turns out that $\|\cdot\|_{p}$ is a norm in $\mathcal{S}_{\M}$ for $1\leq p<\infty$, and a $p$-norm for $0< p<1$. The completion of $(\mathcal{S}_{\M},\|\cdot\|_{p})$ is the noncommutative $L_{p}$ space associated to the pair $(\M,\tau)$, which is simply denoted by $L_{p}(\M)$. As usual, we set $L_{\infty}(\M) = \M$ equipped with the operator norm. 

We also work with noncommutative weak $L_{p}$ spaces. Let $\M'$ be the commutant of $\M$. A closed densely defined operator on $\mathcal{H}$ ($\mathcal{H}$ being the Hilbert space on which $\M$ acts) is said to be affiliated with $\M$ if it commutes with any unitary in $\M'$. Given a densely defined selfadjoint operator $x$, its spectral projection $\int_{\mathcal{I}} d
\gamma_x(\lambda)$ will be simply denoted by $\chi_{\mathcal{I}}(x)$, where $\mathcal{I}$ is a measurable subset of $\R$.
A closed and densely defined operator
$x$ affiliated with $\mathcal{M}$ is said to be \emph{$\tau$-measurable} if
there is $\lambda\in\R_{+}$ such that $$\tau \big( \chi_{(\lambda,\infty)}
(|x|) \big) < \infty.$$
Let $L_{0}(\M)$ be the set of the $\ast$-algebra of \emph{$\tau$-measurable} operators. For $0< p<\infty$, the weak $L_{p}$ space $L_{p,\infty}(\M)$ is defined as the set of all $x$ in $L_0(\M)$ with the following finite quasi-norm
$$\|x\|_{p,\infty}=\sup_{\lambda > 0}\lambda\tau \big( \chi_{(\lambda,\infty)}
(|x|) \big)^{\frac{1}{p}}.$$
The following property has already been proved in 
\cite[Lemma 16]{Su12} that
for any $x_1, x_2 \in
L_{1,\infty}(\M)$ and any $\lambda\in\R_{+}$
\begin{align}\label{distri}
\tau\big((\chi_{(\lambda,\infty)}
(|x_1+x_2|)\big)\leq \tau\big(\chi_{(\lambda/2,\infty)}
(|x_1|)\big)+\tau\big(\chi_{(\lambda/2,\infty)}(|x_2|)\big).
\end{align}

The reader is referred to e.g. \cite{FK,P2} for a comprehensive study of noncommutative $L_p$ spaces.
\subsection{Noncommutative Hilbert-valued $L_p$ spaces}\quad

\medskip

In this subsection, we recall the noncommutative Hilbert-valued $L_p$ spaces \cite{P2}. Let $(x_{n})$ be a finite sequence in $L_{p}(\mathcal M)$. Define
$$\|(x_{n})\|_{L_{p}(\mathcal M; \ell_{2}^{r})}=\|(\sum_{n}|x^{\ast}_{n}|^{2})^{\frac{1}{2}}\|_{p},\ \|(x_{n})\|_{L_{p}(\mathcal M; \ell_{2}^{c})}=\|(\sum_{n}|x_{n}|^{2})^{\frac{1}{2}}\|_{p}.$$
Then ${L_{p}(\mathcal M; \ell_{2}^{r})}$ (resp. ${L_{p}(\mathcal M; \ell_{2}^{c})}$) is defined as the completion of all finite sequences in $L_p(\mathcal M)$ with respect to $\|\cdot\|_{{L_{p}(\mathcal M; \ell_{2}^{r})}}$ (resp. $\|\cdot\|_{{L_{p}(\mathcal M; \ell_{2}^{c})}}$). The space $L_p(\mathcal{M};
\ell_{2}^{rc})$ is defined as follows.
\begin{itemize}
	\item If $2\leq p\leq\infty$,
	$$L_p(\mathcal{M};
	\ell_{2}^{rc})=L_{p}(\mathcal M; \ell_{2}^{c})\cap L_{p}(\mathcal M; \ell_{2}^{r})$$
	equipped with the intersection norm:
	$$\|(x_{n})\|_{L_p(\mathcal{M};
		\ell_{2}^{rc})}=\Big(\|(x_{n})\|^p_{L_{p}(\mathcal M; \ell_{2}^{c})}+\|(x_{n})\|^p_{L_{p}(\mathcal M; \ell_{2}^{r})}\Big)^{\frac1p}.$$
	\item If $1\leq p<2$,
	$$L_p(\mathcal{M};
	\ell_{2}^{rc})=L_{p}(\mathcal M; \ell_{2}^{c})+ L_{p}(\mathcal M; \ell_{2}^{r})$$
	equipped with the sum norm:
	$$\|(x_{n})\|_{L_p(\mathcal{M};
		\ell_{2}^{rc})}=\inf\Big(\|(y_{n})\|^p_{L_{p}(\mathcal M; \ell_{2}^{c})}+\|(z_{n})\|^p_{L_{p}(\mathcal M; \ell_{2}^{r})}\Big)^{\frac1p},$$
	where the infimum runs over all possible decompositions $x_{n}=y_{n}+z_{n}$ with $y_{n}$ and $z_{n}$ in $L_{p}(\mathcal{M})$.
\end{itemize}
 This procedure is also used to define the spaces
$L_{1,\infty}(\mathcal{M}; \ell_{2}^{r})$ (resp.
$L_{1,\infty} (\mathcal{M}; \ell_{2}^{c})$) and  $L_{1,\infty} (\mathcal{M}; \ell_{2}^{rc})$ with the sum norm,
$$\|(x_{n})\|_{L_{1,\infty}(\mathcal{M};
	\ell_{2}^{rc})}=\inf_{x_{n}=y_{n}+z_{n}}\big\{\|(y_{n})\|_{L_{1,\infty}(\mathcal{M}; \ell_{2}^{c})}+\|(z_{n})\|_{L_{1,\infty}(\mathcal{M}; \ell_{2}^{r})}\big\}.$$

We remark that the definition of space ${L_{p}(\mathcal M; \ell_{2}^{rc})}$ equals the classical one~\cite{P2}, and one can easily see that the following basic properties related to space ${L_{p}(\mathcal M; \ell_{2}^{rc})}$ are also valid.

\begin{prop}\label{dual}
Let $1\le p<\8$ and $p^\prime$ be its conjugate index. Then
\begin{equation*}
  (L_p(\M;\ell_2^{c}))^*=L_{p^\prime}(\M;\ell_2^{c}),\quad (L_p(\M;\ell_2^{r}))^*=L_{p^\prime}(\M;\ell_2^{r})
\end{equation*}
and
\begin{equation*}
  (L_p(\M;\ell_2^{rc}))^*=L_{p^\prime}(\M;\ell_2^{rc}).
\end{equation*}
The anti-duality bracket is given by
$$\langle (y_n), (x_n)\rangle=\sum_n\tau(x_ny^*_n),\quad (x_n)\subset L_p(\M), (y_n)\subset L_{p^\prime}(\M).$$
\end{prop}

The following noncommutative Khintchine inequalities will be frequently used. See e.g.  ~\cite{LP86,LP-Pi91,Pis98, C1} for the proof.
\begin{prop}\label{nonkin}
Let $(\varepsilon_n)$ be a sequence of independent Rademarcher random variables on a probability space $(\Omega, P)$. Let $1\le p<\8$ and $(x_n)$ be a sequence in $L_p(\M;\ell^{rc}_2)$. For $1\le p<\8$, there exist two positive constants $c_p$ and $C_p$ such that
\begin{equation*}
  c_p\|(x_n)\|_{L_p(\mathcal M;\ell^{rc}_2)}\le \bigg\|\sum_{n}\varepsilon_nx_n\bigg\|_{L_p(L_{\infty}(\Omega)\overline{\otimes}
\mathcal M)}\le  C_p\|(x_n)\|_{L_p(\mathcal M;\ell^{rc}_2)}.
\end{equation*}
The above estimate is still true if one replaces the $L_p$ spaces by the weak $L_{1}$ space. 


\end{prop}

\subsection{Noncommutative Calder{\'o}n-Zygmund decomposition}\quad

\medskip

In this subsection, we introduce the  noncommutative Calder{\'o}n-Zygmund decomposition developed in \cite{CCP}, whose construction is based on the noncommutative martingale theory. To this end, we first introduce the related notions. For each $n\in\N$, let $\I_n$ be the set of dyadic intervals with length of $2^n$ in $\mathbb{Z}$, that is each interval in $\I_n$ can be written as $[s2^n,(s+1)2^n)$, where $s$ is an integer.
Let $\sigma_n$ be the $n$-th $\sigma$-algebra generated by $\I_n$ and $\mathcal N_n=L_\infty(\mathbb \Z,\sigma_n)\overline{\otimes}\M$. Recall that $\mathcal N=L_\infty(\mathbb Z)\overline{\otimes}\M$. Then $(\mathcal{N}_n)_{n \in \N}$ is a sequence of decreasing von Neumann subalgebras of $\mathcal N$. 
Hence, $(\mathcal{N}_n)_{n \in \N}$ forms a filtration and the resulting conditional expectations $(\mathsf{E}_n)_{n\in\N}$ satisfy
\begin{equation*}
  \forall~m,n\in\N,\quad \mathsf{E}_m\mathsf{E}_n=\mathsf{E}_n\mathsf{E}_m=\mathsf{E}_{\max(m,n)}
\end{equation*}
 and for $f\in L_p(\mathcal N)$ with $1\le p<\8$,
 \begin{equation}\label{martingale}
  f_{n} :=\mathsf{E}_n(f)=\sum_{I \in \I_n}^{\null} f_I \chi_I,
 \end{equation}
where $\chi_I$ is the characteristic function of $I$ and
$$f_I = \frac{1}{|I|} \sum_{y\in I} f(y).$$
It is easy to check that $(f_{n})_{n\in\N}$ is a $L_p$-reverse martingale, namely $\sup_{n\in\N}\|f_{n}\|_{L_p(\mathcal N)}<\8$.
The resulting martingale difference sequence $df=(df_{n})_{n\in\N}$ is defined by $df_{n}=f_{n-1}-f_{n}$ for $n\ge 1$ and $df_0=0$.

To give the content of noncommutative Calder{\'o}n-Zygmund decomposition, consider
$$\mathcal N_{c,+} =\Big\{
f: \Z \to \M\cap L_1(\M) \, \big| \ f \geq0, \
\overrightarrow{\mathrm{supp}} \hskip1pt f \ \ \mathrm{is \
finite} \Big\},$$
which is dense in $L_1(\mathcal N)_{+}$. Here
$\overrightarrow{\mathrm{supp}}f=\mathrm{supp}\|f\|_{_{L_{1}(\M)}}$. Observe that for any given $f \in \mathcal{N}_{c,+}$ and $\lambda>0$, there exists $m_{\lambda}(f)\in\N$ such that $f_{n}\leq\lambda\1_{\mathcal{N}}$ for all $n\geq m_{\lambda}(f)$ (see \cite[Lemma 3.1]{JP1}), where $\1_{\mathcal{N}}$ denotes the unit
element in $\mathcal{N}$.

The following modified Cuculescu's theorem \cite{Cuc} was obtained in \cite[Lemma 3.1]{JP1}.
\begin{lem}\label{cucu}
Let $f \in \mathcal{N}_{c,+}$ and consider its related dyadic martingale $(f_n)_{n\in\N}$. Given $\lambda>0$, there exists an increasing sequence of projections $(q_n)_{n\in\N}$ defined by $q_n = \1_\mathcal{N}$ for $n\ge m_{\lambda}(f)$ and recursively for $n< m_{\lambda}(f)$
$$q_n=q_n(f,\lambda)=\chi_{(0,\lambda]}(q_{n+1} f_n q_{n+1})$$
such that the following conclusions hold:
\begin{enumerate}[\noindent]
	\item~\emph{(i)} $q_n$ commutes with $q_{n+1} f_n
	q_{n+1}$ for each $n$;
	\item~\emph{(ii)} $q_n$ belongs to $\mathcal{N}_n$ and $q_n f_n q_n \le \lambda \hskip1pt
	q_n$ for each $n$;
	\item~\emph{(iii)} set $q_0:=q=\bigwedge_{n=0}^{m_{\lambda}(f)} q_n$, then $\lambda\varphi(
	\mathbf{1}_\mathcal{N} - q) \le\|f\|_1.$  
\end{enumerate}
\end{lem}
In fact, for each $n$, $q_n$ admits the following expression (see e.g. \cite{JP1})
$$q_{n}=\sum_{I\in\I_{n}}q_{I}\chi_{I},$$
with $q_{I}$ projections in $\M$ defined by
$$q_{I}=\begin{cases} \1_\M & \mbox{if} \ n\ge m_{\lambda}(f),
\\ \chi_{(0,\lambda]} \big( q_{\widehat{I}} f_I q_{\widehat{I}}\big) & \mbox{if} \ 0\le n < m_{\lambda}(f),
\end{cases}$$
where $\widehat{I}$ is the dyadic father of $I$. Accordingly, these projections satisfy
\begin{equation}\label{czd5}
\  q_I\leq q_{\widehat{I}},\ \
\  q_I\  \mbox{commutes\ with}\  q_{\widehat{I}} f_I q_{\widehat{I}},\ \
\   q_I f_I q_I \le \lambda q_I.
\end{equation}
If we define the sequence $(p_n)_{n}$ of pairwise disjoint projections by
\begin{equation}\label{czd1}
p_n=q_{n+1}-q_n=\sum_{I\in\I_{n}}(q_{\widehat{I}}-q_{I})\chi_{I}\triangleq\sum_{I\in\I_{n}}p_I\chi_{I}
\end{equation}
for each $n$, then
\begin{equation}\label{346679}
\sum_{n} p_n = \1_\mathcal{N} - q =q^\perp
\end{equation}
and for each $n$
\begin{equation}\label{czd1123}
\|p_nf_{n}p_n\|_{\infty}\leq2\lambda.
\end{equation}

Based on the previous notation, the noncommutative analogue for the Calder{\'o}n-Zygmund decomposition was recently found by Caldilhac et al \cite{CCP}.
\begin{prop}\label{czdecom}
Fix $f\in\mathcal N_{c,+}$ and $\lambda>0$. Let $(q_n)_{n}$ and $(p_n)_{n}$ be the two sequences of projections appeared in the above Cuculescu's construction. Then there exist a projection $\zeta\in \mathcal N$ defined by
\begin{equation}\label{czd9}
\zeta = \big(\bigvee_{I\in \I} p_I\chi_{5I}\big)^{\bot},
\end{equation}
where $5I$ denotes the interval with the same center as $I$ with length $|5I|=5|I|$,
and a decomposition of $f$,
\begin{equation}\label{czd76789}
f = g + b
\end{equation}
such that the following assertions hold.

\begin{enumerate}[\noindent]
\item~\emph{(i)} $\lambda\varphi(\mathbf 1_\mathcal{N}-\zeta) \leq 5\|f\|_1$.
\item~\emph{(ii)} $g=qfq+\sum_{n}p_{n}f_{n}p_{n}$ satisfies
$\| g\|_1 \le\|f\|_1 \quad \mbox{and} \quad \|
g\|_\infty \le 2\lambda$.
\item~\emph{(iii)} $b=\sum_{n}b_{n}$, where
\begin{equation}\label{czd6}
b_{n}=p_n (f-f_{n}) q_n+q_{n+1}(f-f_{n})p_n.
\end{equation}
Each $b_{n}$ satisfies two cancellation conditions:
     $\mathsf{E}_n b_{n} = 0$; and
    for all $x,y\in\Z$ with $y \in 5I_{x,n}$, $\zeta(x)b_{n}(y)\zeta(x) = 0$, where $I_{x,n}$ is the unique interval in $\I_n$ containing $x$.
\end{enumerate}
\end{prop}

\section{Proof of Theorem~\ref{Th5}: one reduction}\label{redu}\quad
To prove Theorem \ref{Th5}, we give a reduction in the present section. Motivated by the study of the variational inequalities \cite{Bour89}, we split the square function into the `long one' and the `short one'. To be more precise, fix an increasing sequence $(n_{i})_{i\in\N}$ and let $(A_{n_i})_{i\in\N}$ be the associated nested sequence. For an interval $I_i=[n_{i},n_{i+1})$, one can see that there are two cases:
\begin{enumerate}[(1)]
	\item [$\bullet$] Case 1: $I_i$ contains no dyadic point, that is, for any $k\in\mathbb N$, $2^k\notin I_i$;
	\item [$\bullet$] Case 2: $I_i$ contains at least one dyadic point $2^k$ for $k\in\N$.
\end{enumerate}
According to the above classification, for each interval $I_i=[n_{i},n_{i+1})$, we split it into at most three disjoint parts
\begin{equation}\label{decomposition}
  I_i:=[n_i,\tilde{n}_i)\cup[\tilde{n}_i, \tilde{\tilde{n}}_{i})\cup[{\tilde{\tilde{n}}}_{i},n_{i+1})
\end{equation}
by the law: if $I_i$ belongs to Case 1, then set $\tilde{n}_i=\tilde{\tilde{n}}_{i}=n_{i+1}$; if $I_i$ belongs to Case 2, we set
$\tilde{n}_i=2^{k_{i}}:=\min\{2^{k}:2^{k}\in I_i\}$ and $\tilde{\tilde{n}}_{i}=2^{l_{i}}:=\max\{2^{k}:2^{k}\in \bar{I_i}\}$ where $\bar{I_i}$ is the closure of $I_i$.

By above decomposition of intervals and using the quasi-triangle inequality for weak $L_1$ norm $\|\cdot\|_{L_{1,\8}(\mathcal{N}; \ell_{2}^{rc})}$, we have
\begin{equation}\label{long-short}
\begin{split}
  &\|(M_{A_{n_{i}}}f-M_{A_{n_{i+1}}}f)_{i\in\N}\|_{L_{1,\8}(\mathcal{N};\ell_{2}^{rc})}\\
  &\le 3\|(M_{A_{n_{i}}}f-M_{A_{\tilde{n}_i}}f)_{i\in\N}\|_{L_{1,\8}(\mathcal{N};\ell_{2}^{rc})}+3\|(M_{A_{\tilde{n}_{i}}}f
  -M_{A_{\tilde{\tilde{n}}_{i}}}f)_{i\in\N}\|_{L_{1,\8}(\mathcal{N};\ell_{2}^{rc})}\\
  &\ \ \ \ +3\|(M_{A_{\tilde{\tilde{n}}_{i}}}f-M_{A_{n_{i+1}}}f)_{i\in\N}\|_{L_{1,\8}(\mathcal{N};\ell_{2}^{rc})}.
\end{split}
\end{equation}

On the other hand, by~\eqref{decomposition}, we introduce two collections of intervals with respect to $\{[n_i,n_{i+1})\}_i$:
\begin{enumerate}[(1)]
\item [$\bullet$] ${\mathrm S}$ consists of all intervals $I_i$ belonging to Case 1, or $[{n}_{i},{\tilde{n}}_{i})$, $[{\tilde{\tilde{n}}}_{i},n_{i+1})$ in~\eqref{decomposition}. 
 	\item [$\bullet$] ${\mathrm L}$ consists of all intervals $[\tilde{n}_i, \tilde{\tilde{n}}_{i})$ in~\eqref{decomposition}.
\end{enumerate}
It is not difficult to check that $\mathrm L\cup \mathrm S$ is a disjoint family of intervals and forms a finer partition of $\big\{[n_{i},n_{i+1})\big\}_{i\in\mathbb N}$. By~\eqref{long-short}, we have
\begin{equation}\label{lo-sh}
\begin{split}
  &\|(M_{A_{n_{i}}}f-M_{A_{n_{i+1}}}f)_{i\in\N}\|_{L_{1,\8}(\mathcal{N};\ell_{2}^{rc})}\\
  &\le
  3\|(M_{A_{\tilde{n}_i}}f-M_{A_{\tilde{\tilde{n}}_{i:}}}f)_{i:[\tilde{n}_i,\tilde{\tilde{n}}_{i})\in L}\|_{L_{1,\8}(\mathcal{N};\ell_{2}^{rc})}\\
  &\ \ +6\|(M_{A_{m_{i}}}f-M_{A_{\widetilde{m}_{i}}}f)_{i:[m_i,\widetilde{m}_i)\in\mathrm S}\|_{L_{1,\8}(\mathcal{N};\ell_{2}^{rc})}.
\end{split}
\end{equation}
We now focus on the first term on the right hand side of the above inequality. Fix an interval $[\tilde{n}_i,\tilde{\tilde{n}}_{i})$. By~\eqref{decomposition}, we write $[\tilde{n}_i,\tilde{\tilde{n}}_{i})=[2^{k_{i}},2^{l_{i}})$. Decompose 	
\begin{align*}
 (M_{A_{2^{k_{i}}}}f-M_{A_{2^{l_{i}}}}f)=(M_{A_{2^{k_{i}}}}f-\mathsf{E}_{k_i}f)
 +(\mathsf{E}_{k_i}f-\mathsf{E}_{l_{i}}f)+(\mathsf{E}_{l_{i}}f-M_{A_{2^{l_i}}}f),
\end{align*}
where $(\mathsf{E}_{k})_k$ is the dyadic conditional expectations defined in the preliminary section. As a consequence, there exists a sequence of positive integers $k_0<l_0\leq k_1<l_1< \dotsm\leq k_i<l_i\leq\dotsm$ such that
\begin{equation*}
\begin{split}
  &\|(M_{A_{\tilde{n}_i}}f-M_{A_{\tilde{\tilde{n}}_{i}}}f)_{i:[\tilde{n}_i,\tilde{\tilde{n}}_{i})\in\mathrm L}\|_{L_{1,\8}(\mathcal{N};\ell_{2}^{rc})}\le3\|(M_{A_{2^{k_{i}}}}f-\mathsf{E}_{k_i}f)_{{i}}\|_{L_{1,\8}(\mathcal{N};\ell_{2}^{rc})}\\
  &+3\|(\mathsf{E}_{k_i}f-\mathsf{E}_{l_{i}}f)_{{i}}\|_{L_{1,\8}(\mathcal{N};\ell_{2}^{rc})}
  +3\|(M_{A_{2^{l_{i}}}}f-\mathsf{E}_{l_{i}}f)_{{i}}\|_{L_{1,\8}(\mathcal{N};\ell_{2}^{rc})}.
\end{split}
\end{equation*}
Using the fact that $n\mapsto(\sum_{k=1}^n|x_k|^2)^{\frac12}$ is increasing, one easily checks that
\begin{equation}\label{long}
\begin{split}
  &\|(M_{A_{\tilde{n}_i}}f-M_{A_{\tilde{\tilde{n}}_{i}}}f)_{i:[\tilde{n}_i,\tilde{\tilde{n}}_{i})\in\mathrm L}\|_{L_{1,\8}(\mathcal{N};\ell_{2}^{rc})}\\
  &\le3\|(\mathsf{E}_{k_i}f-\mathsf{E}_{l_{i}}f)_{{i}}\|_{L_{1,\8}(\mathcal{N};\ell_{2}^{rc})}+6\|(M_{A_{2^k}}f-\mathsf{E}_{k}f)_{k\in\Z}\|_{L_{1,\8}(\mathcal{N};\ell_{2}^{rc})}.
\end{split}
\end{equation}
Together with~\eqref{lo-sh} and~\eqref{long}, we obtain
\begin{equation}\label{divide}
\begin{split}
  &\|(M_{A_{n_{i}}}f-M_{A_{n_{i+1}}}f)_{i\in\N}\|_{L_{1,\8}(\mathcal{N};\ell_{2}^{rc})}\\
  &\le 6\|(M_{A_{m_{i}}}f-M_{A_{\widetilde{m}_{i}}}f)_{i:[m_i,\widetilde{m}_i)\in\mathrm S}\|_{L_{1,\8}(\mathcal{N};\ell_{2}^{rc})}\\
  &\ \ \ +9\|(\mathsf{E}_{k_i}f-\mathsf{E}_{l_{i}}f)_{{i}}\|_{L_{1,\8}(\mathcal{N};\ell_{2}^{rc})}+18\|(M_{A_{2^k}}f-\mathsf{E}_{k}f)_{k\in\Z}\|_{L_{1,\8}(\mathcal{N};\ell_{2}^{rc})}.
\end{split}
\end{equation}
\begin{remark}\label{ture}{\rm
By applying the same arguments as above,  it is clear that there exist similar dominations as \eqref{divide} for $\|\cdot\|_{L_p(\mathcal{N}; \ell_{2}^{rc})}$ and $\|\cdot\|_{\mathrm{BMO}_{d}(\mathcal{R})}$ via the triangle inequalities with possibly different constants.}
\end{remark}

Concerning the sequence $(M_{A_{2^k}}f-\mathsf{E}_{k}f)_{k\in\Z}$, the first and third authors~\cite{HX} established the following results.
\begin{lem}[\cite{HX}]\label{long-ineq}
 Let $1\le p\le \8$. Set $L_kf=M_{A_{2^k}}f-\mathsf{E}_{k}f$. Then the following assertions are true with a
positive constant $C_p$ depending only on $p$:
\begin{enumerate}[\noindent]
\item~\emph{(i)} for $p=1$, $$\|(L_kf)_{k\in\Z}\|_{L_{1,\infty}(\mathcal{N};
\ell_{2}^{rc})}\le  C_{p}\|f\|_{1},\; \forall f\in L_{1}(\mathcal N);$$
\item~\emph{(ii)} for $p=\infty$,
\begin{align*}
\Big\|\sum_{{k\in\Z}} L_kf\otimes e_{1k}
\Big\|_{\mathrm{BMO}_{d}(\mathcal{R})}+
\Big\|\sum_{{k\in\Z}} L_kf \otimes e_{k1}
\Big\|_{\mathrm{BMO}_{d}(\mathcal{R})}\le  C_{p}\,\|f\|_\infty,\; \forall f\in L_{\infty}(\mathcal N);
\end{align*}
\item~\emph{(iii)} for $1<p<\infty$, $$\|(L_kf)_{k\in\Z}\|_{L_p(\mathcal{N};
\ell_{2}^{rc})}\le  C_{p} \|f\|_{p}, \; \forall f\in L_{p}(\mathcal N).$$
\end{enumerate}
\end{lem}

On the other hand, $(\mathsf{E}_{k_i}f-\mathsf{E}_{l_{i}}f)_{{i}}$ forms a new sequence of martingale differences, and the dyadic martingale analogue of Lemma \ref{long-ineq} have been established in \cite{Ran1}. See also \cite{PX, Ran1,Ran2} for more on noncommutative Burkholder-Gundy inequalities.

Hence, we give our main efforts to the  first term on the right hand side of \eqref{divide}, namely the sequence $(M_{A_{m_{i}}}f-M_{A_{\widetilde{m}_{i}}}f)_{i:[m_i,\widetilde{m}_i)\in\mathrm S}$. We denote, by abuse of notation, the sequence $\{m_0,\widetilde{m}_0,m_1,\dotsm, m_i,\widetilde{m}_i,\dotsm\}$ as  $\{m_0,m_1, m_2\dotsm, m_i,m_{i+1},\dotsm\}$. We abbreviate each interval $[s_i,s_{i+1})$ to $i$ and denote the collection of such $i$ by $\mathcal S$. Setting $T_if=M_{A_{m_{i}}}f-M_{A_{m_{i+1}}}f$, and thus we get $(T_if)_{i\in\mathcal S}$ standing for
 $(M_{A_{m_{i}}}f-M_{A_{\widetilde{m}_{i}}}f)_{i:[m_i,\widetilde{m}_i)\in\mathrm S}$.
 Combining \eqref{divide}, Remark~\ref{ture}, Lemma~\ref{long-ineq} with the fact that $n\mapsto(\sum_{k=1}^n|x_k|^2)^{\frac12}$ is increasing, to establish Theorem \ref{Th5}, it suffices to show the following result.

\begin{thm}\label{t5343232}
Let $\mathcal S$ and $(T_if)_{i\in\mathcal S}$ be defined as above. Let $1\le p\le \8$. Then the following assertions are true with a
positive constant $C_p$ depending only on $p$:
\begin{enumerate}[\noindent]
\item~\emph{(i)} for $p=1$, $$\|(T_{i}f)_{i\in\mathcal{S}}\|_{L_{1,\infty}(\mathcal{N};
\ell_{2}^{rc})}\le  C_{p}\|f\|_{1},\; \forall f\in L_{1}(\mathcal N);$$
\item~\emph{(ii)} for $p=\infty$, $$\Big\|
\sum_{i\in\mathcal{S}} T_{i}f \otimes e_{1i}
\Big\|_{\mathrm{BMO}_{d}(\mathcal{R})} +
\Big\|\sum_{i\in\mathcal{S}} T_{i}f \otimes e_{i1}
\Big\|_{\mathrm{BMO}_{d}(\mathcal{R})} \le  C_{p}\,
\|f\|_\infty,\; \forall f\in L_{\infty}(\mathcal N);$$
\item~\emph{(iii)} for $1<p<\infty$, $$\| (T_{i}f)_{i\in\mathcal{S}}\|_{L_p(\mathcal{N};
\ell_{2}^{rc})}\le  C_{p} \|f\|_{p}, \; \forall f\in L_{p}(\mathcal N).$$
\end{enumerate}
\end{thm}

\section{Proof of Theorem \ref{t5343232}: Weak type $(1,1)$ estimate}\label{weak}
In this section, we show the weak type $(1,1)$ estimate stated in Theorem \ref{t5343232}.
\subsection{Some technical lemmas}\quad

\medskip

Recall that $(\varepsilon_{i})$ is a Rademacher sequence on a fixed probability space $(\Omega,P)$. Define
\begin{equation}\label{9898}
{T}f(x)=\sum_{i\in\mathcal{S}}\varepsilon_{i}T_if(x)=\sum_{i\in\mathcal{S}}\varepsilon_{i}(M_{A_{n_{i}}}-M_{A_{n_{i+1}}})f(x).
\end{equation}
Then Proposition~\ref{nonkin}(ii) immediately implies the following result.
\begin{lem}\label{inte1234}
Let $h\in \mathcal{N}_{c,+}$. Then
$$\|(T_{i}h)_{i\in\mathcal{S}}\|_{L_{1,\infty}(\mathcal{N};\ell_{2}^{rc})}\thickapprox
\|{T}h\|_{L_{1,\infty}(L_{\infty}(\Omega)\overline{\otimes}\mathcal N)}.$$
\end{lem}
\begin{lem}\label{L1}\rm
Let $(S_{k,i})_{k,i\in\Z}$ be a sequence of bounded linear operators on $L_2(\mathcal{N})$. Let $h\in L_2(\mathcal{N})$.
If $(u_{n})_{n\in\Z}$ and $(v_{n})_{n\in\Z}$ are two sequences of operators in $L_{2}(\mathcal{N})$ such that $h=\sum_{n\in\Z}u_{n}$ and $\sum_{n\in\Z}\|v_{n}\|_{2}^{2}<\infty$, then
$$\sum_{k\in\Z}\|(S_{k,i}h)_{i}\|_{L_2(\mathcal{N};\ell_{2}^{rc})}^{2}\leq w^{2}\sum_{n\in\Z}\|v_{n}\|_{2}^{2}$$
provided that there exists a sequence $(\sigma(j))_{j\in\Z}$ of positive numbers with $w=\sum_{j\in\Z}\sigma(j)<\infty$ such that
\begin{align}\label{HLX917}
\|(S_{k,i}u_{n})_{i}\|_{L_2(\mathcal{N};\ell_{2}^{rc})}\leq \sigma(n-k)\|v_{n}\|_{2}
\end{align}
for every $n,k$.
\end{lem}
\begin{proof}
By the triangle inequality in $L_2(\mathcal{N};\ell_{2}^{rc})$, (\ref{HLX917}) and the Young inequality in $\ell_{2}$, we deduce that
\begin{align*}
\sum_{k\in\Z}\|(S_{k,i}h)_{i}\|_{L_2(\mathcal{N};\ell_{2}^{rc})}^{2}
\leq&\ \hskip1pt \sum_{k\in\Z}\bigg(\sum_{n\in\Z}\|(S_{k,i}u_{n})_{i}\|_{L_2(\mathcal{N};\ell_{2}^{rc})}\bigg)^{2}\\
\leq&\ \sum_{k\in\Z}\bigg(\sum_{n\in\Z}\sigma(n-k)\|v_{n}\|_{2}\bigg)^{2}\\
\leq&\ \bigg(\sum_{n\in\Z}\sigma(n)\bigg)^{2}\bigg(\sum_{k\in\Z}\|v_{k}\|_{2}^{2}\bigg),
\end{align*}
which finishes the proof.
\end{proof}

Let $A$ be a subset of $\Z$, define
\begin{align}\label{bad347237289}
	\mathcal{I}(A,n)=\bigcup_{{\begin{subarray}{c}
				I\in\I_{n} \\ \partial A\cap I\neq \emptyset
	\end{subarray}}} A\cap I
\end{align}
and
\begin{align}\label{bad47237289}
	\mathcal{I}_1(A,n)=\bigcup_{{\begin{subarray}{c}
				I\in\I_{n} \\ \partial A\cap I\neq \emptyset
	\end{subarray}}} I,
\end{align}
where $\partial A$ means the boundary of $A$.

\begin{lem}\label{KeyLem}
Let $h\in L_2({\mathcal{N}})$. Then for any $x\in\Z$, $n\in\mathbb N$ and subset $B\subseteq\Z$,
\begin{align*}
  &\Big\|\sum_{y\in \mathcal{I}(x+B,n)}h(y)\Big\|^2_{L_2(\M)}\le |\mathcal{I}_1(x+B,n)|\sum_{y\in \mathcal{I}_1(x+B,n)}\|\mathsf{E}_n(h\1_{x+B})(y)\|^2_{L_2(\M)};\\
  &\Big\|\sum_{y\in \mathcal{I}(x+B,n)}h(y)\Big\|^2_{L_2(\M)}\le |\mathcal{I}(x+B,n)|\sum_{y\in \mathcal{I}(x+B,n)}\|h(y)\|^2_{L_2(\M)}.
\end{align*}
\end{lem}
\begin{proof}
We first prove the first inequality. It is easy to check that
\begin{equation*}
  \sum_{y\in \mathcal{I}(x+B,n)}h(y)=\sum_{\substack{I\in\I_n\\ I\cap\partial(x+B)\neq\emptyset}}\sum_{y\in I}\mathsf{E}_n(h\1_{x+B})(y)=
  \sum_{y\in \mathcal{I}_1(x+B,n)}\mathsf{E}_n(h\1_{x+B})(y).
\end{equation*}
By the Minkowski and Cauchy-Schwarz inequalities, we obtain
\begin{align*}
  \Big\|\sum_{y\in \mathcal{I}(x+B,n)}h(y)\Big\|^2_{L_2(\M)}&\le \Big(\sum_{y\in \mathcal{I}_1(x+B,n)}\|\mathsf{E}_n(h\1_{x+B})(y)\|_{L_2(\M)}\Big)^2\\
  &\le|\mathcal{I}_1(x+B,n)|\sum_{y\in \mathcal{I}_1(x+B,n)}\|\mathsf{E}_n(h\1_{x+B})(y)\|^2_{L_2(\M)}.
\end{align*}
The same argument gives
\begin{align*}
  \Big\|\sum_{y\in \mathcal{I}(x+B,n)}h(y)\Big\|^2_{L_2(\M)}&\le \Big(\sum_{y\in \mathcal{I}(x+B,n)}\|h(y)\|_{L_2(\M)}\Big)^2\\
  &\le|\mathcal{I}(x+B,n)|\sum_{y\in \mathcal{I}(x+B,n)}\|h(y)\|^2_{L_2(\M)}.
\end{align*}
This completes the proof.
\end{proof}

\bigskip

Now we are ready to prove the weak type $(1,1)$ estimate in Theorem \ref{t5343232}. Note that we just consider the case that each $A_{n_i}$ is written as $[0,n_i]$, since another case $A_{n_{i}}$ of the form $[-n_i,n_i]$ can be handled in the same way.

By decomposing $f = f_{1}-f_{2} +i(f_{3}-f_{4})$ with $f_{j}\geq0$ such that $\|f_{j}\|_{1}\leq\|f\|_{1}$ for $j=1,2,3,4$, we may assume that $f$ is positive. Moreover, since $\mathcal N_{c,+}$ is dense in $L_1(\mathcal N)_{+}$, by the standard approximation argument, it suffices to consider $f\in \mathcal N_{c,+}$. Now fix one $f\in \mathcal N_{c,+}$ and a $\lambda\in(0,+\infty)$.
Using Theorem \ref{czdecom}, we can decompose $f$ as $f=g+b$. Then the distribution inequality gives (\ref{distri}),
$$\widetilde{\varphi}\big(\chi_{(\lambda,\infty)}(|Tf|)\big)\leq \widetilde{\varphi}\big(\chi_{(\lambda/2,\infty)}(|Tg|)\big)+\widetilde{\varphi}\big(\chi_{(\lambda/2,\infty)}(|Tb|)\big),$$
where $\widetilde{\varphi}=\int_{\Omega}\otimes\varphi$.
Therefore, by Lemma \ref{inte1234}, it suffices to show
\begin{align}\label{HLX7}
\widetilde{\varphi}(\chi_{(\lambda/2,\infty)}(|Tb|))\lesssim\frac{\| f\|_{1}}{\lambda},
\end{align}
\begin{align}\label{HLX15}
\widetilde{\varphi}(\chi_{(\lambda/2,\infty)}(|Tg|))\lesssim\frac{\| f\|_{1}}{\lambda}.
\end{align}

\subsection{Weak type estimate for the bad function: (\ref{HLX7})}\quad

\medskip

Using the projection $\zeta$ introduced in Proposition~\ref{czdecom}, we decompose $Tb$  as
$$Tb = (\1_\mathcal{N} - \zeta) T
b (\1_\mathcal{N} - \zeta) + \zeta \hskip1pt T b (\1_\mathcal{N} - \zeta)
+ (\1_\mathcal{N} - \zeta)T b \zeta + \zeta \hskip1pt T b
\zeta.$$
In particular, by Proposition~\ref{czdecom}(i), we find
\begin{align*}
\hskip1pt \widetilde{\varphi} \big(\chi_{(\lambda/2,\infty)}(|Tb|)
\big)
\lesssim&\ \hskip1pt \varphi (\1_\mathcal{N} - \zeta) +
\hskip1pt \widetilde{\varphi} \big(\chi_{(\lambda/8,\infty)}(|\zeta Tb\zeta|)
\big)\\
\lesssim&\ \frac{ \|f\|_{1}}{\lambda}+\widetilde{\varphi} \big(\chi_{(\lambda/8,\infty)}(|\zeta Tb\zeta|)
\big).
\end{align*}
Hence, we are reduced to showing
\begin{align*}
\widetilde{\varphi} \big(\chi_{(\lambda/8,\infty)}(|\zeta Tb\zeta|)\big)\lesssim\frac{ \|f\|_{1}}{\lambda}.
\end{align*}
Note that the Chebychev inequality gives
$$\lambda^{2}\widetilde{\varphi} \big(\chi_{(\lambda/8,\infty)}(|\zeta Tb\zeta|)\big)\lesssim\|\zeta Tb\zeta\|^{2}_{L_{2}(L_{\infty}(\Omega)\overline{\otimes}\mathcal N)}.$$
Hence, it is enough to prove
\begin{eqnarray}\label{9121}
	\|\zeta Tb\zeta\|^{2}_{L_{2}(L_{\infty}(\Omega)\overline{\otimes}\mathcal N)}\lesssim\lambda^{2}\sum_{n}\|p_{n}\|^{2}_{2},
\end{eqnarray}
due to Cuculescu's construction and (\ref{346679}), $$\sum_{n}\|p_{n}\|^{2}_{2}=\sum_{n}\|p_{n}\|_{1}\lesssim\frac{\|f\|_{1}}{\lambda}.$$
To estimate (\ref{9121}), we first use the orthogonality of $\varepsilon_{i}$ to get
\begin{eqnarray*}\label{333}
	\|\zeta Tb\zeta\|^{2}_{L_{2}(L_{\infty}(\Omega)\overline{\otimes}\mathcal N)}=
	\sum_{i\in\mathcal{S}}
	\|\zeta \, T_i \hskip-1pt b \, \zeta
	\|^{2}_{2}=\sum_{i\in\mathcal{S}}\|
	\zeta \, (M_{A_{n_{i}}}-M_{A_{n_{i+1}}}) \hskip-1pt b \, \zeta\|^{2}_{2}.
\end{eqnarray*}
For each $k\in\Z$, let $\mathcal{S}_k$ be the set of $i$ such that $[n_i,n_{i+1})\subseteq [2^k,2^{k+1})$. Then clearly $\mathcal{S}=\cup_{k\in\Z}\mathcal{S}_k$ since for $k<0$, $\mathcal{S}_k$ is empty. With this convention, we deduce that
$$\sum_{i\in \mathcal{S}}\|
\zeta \, (M_{A_{n_{i}}}-M_{A_{n_{i+1}}}) \hskip-1pt b \, \zeta\|^{2}_{2}=\sum_{k}\sum_{i\in \mathcal{S}_{k}}\|
\zeta \, (M_{A_{n_{i}}}-M_{A_{n_{i+1}}}) \hskip-1pt b \, \zeta\|^{2}_{2}.$$
Hence, (\ref{9121}) is reduced to showing
\begin{eqnarray}\label{bad}
	\sum_{k}\sum_{i\in \mathcal{S}_{k}}\|
	\zeta \, (M_{A_{n_{i}}}-M_{A_{n_{i+1}}}) \hskip-1pt b \, \zeta\|^{2}_{2}\lesssim\lambda^{2}\sum_{n}\|p_{n}\|^{2}_{2}.
\end{eqnarray}

To show \eqref{bad}, by noting the definition of $m_{\lambda}(f)$, we can express $b$ as $b=\sum_{n\leq m_{\lambda}(f)}b_{n}$, where $b_{n}=p_n (f-f_{n}) q_n+q_{n+1}(f-f_{n})p_n$ as in (\ref{czd6}). On the other hand, let
$(S_{k,i}h)_{i\in \mathcal{S}_{k}}=(\zeta(M_{A_{n_{i}}}-M_{A_{n_{i+1}}})b\zeta)_{i\in \mathcal{S}_{k}}$, $u_{n}=b_{n}$ and $v_{n}=p_{n}$ in Lemma \ref{L1}. Then  it suffices to show
\begin{eqnarray}\label{bad1}
	\sum_{i\in \mathcal{S}_{k}}\|\zeta(M_{A_{n_{i}}}-M_{A_{n_{i+1}}})b_{n}\zeta\|^{2}_{2}\lesssim2^{-|k-n|}\lambda^{2}\|p_{n}\|^{2}_{2}.
\end{eqnarray}


In the following, we divide the proof of (\ref{bad1}) into several steps.
\begin{lem}\label{badfunction1}
Fix $i\in \mathcal{S}_{k}$. Then for $n\ge k$,
$$
	\zeta(x) (M_{A_{n_{i}}}-M_{A_{n_{i+1}}}) b_{n}(x)\zeta(x)=0,\;\forall x\in\Z.
$$
\end{lem}
\begin{proof}
This follows from the observation  $\zeta(x)M_{A_{n_{i}}} b_{n}(x)\zeta(x)=\zeta(x)M_{A_{n_{i+1}}} b_{n}(x)\zeta(x)=0,\;\forall x\in\Z$. To see this, fix one $x\in\mathbb Z$. The cancellation property announced in Proposition~\ref{czdecom}(iii) implies
\begin{align*}
	\zeta(x)M_{A_{n_{i+1}}}b_{n}(x)\zeta(x)&=\zeta(x)\frac1{|A_{n_{i+1}}|}\sum_{y\in x+A_{n_{i+1}}}b_{n}(y)\1_{y\notin 5I_{x,n}}\zeta(x)=0
\end{align*}
since $x+A_{n_{i+1}}\subset 5I_{x,n}$ and $k\le n$. The same reasoning implies $\zeta(x)M_{A_{n_{i}}}b_{n}(x)\zeta(x)=0$. This finishes the proof.
\end{proof}
With Lemma \ref{badfunction1}, since $\zeta$ is a projection, it suffices to show for $n< k$
\begin{eqnarray}\label{bad12343}
\sum_{i\in \mathcal{S}_{k}}\|(M_{A_{n_{i}}}-M_{A_{n_{i+1}}})b_{n}\|^{2}_{2}\lesssim2^{n-k}\lambda^{2}\|p_{n}\|^{2}_{2}.
\end{eqnarray}
To show (\ref{bad12343}), by applying the Minkowski inequality, we have
\begin{align*}
\begin{split}
\sum_{i\in \mathcal{S}_{k}}\|(M_{A_{n_{i}}}-M_{A_{n_{i+1}}})b_{n}\|^{2}_{2}&\lesssim\sum_{i\in \mathcal{S}_{k}}\sum_{x\in\Z}
\Big\|\frac{1}{|A_{n_{i+1}}|}\sum_{y\in x+A_{n_{i+1}}\setminus A_{n_{i}}}b_{n}(y)\Big\|^{2}_{L_{2}(\mathcal{M})}\\
&\quad+\sum_{i\in \mathcal{S}_{k}}\Big(\frac{1}{|A_{n_{i}}|}-\frac{1}{|A_{n_{i+1}}|}\Big)^{2}
\sum_{x\in\Z}\Big\|\sum_{y\in x+A_{n_{i}}}b_{n}(y)\Big\|^{2}_{L_{2}(\mathcal{M})}\\
&\triangleq I^{1}_{k,n}+I^{2}_{k,n}.
\end{split}
\end{align*}

\subsubsection{{\bf{Estimate of}} $I^{1}_{k,n}$}
\begin{lem}\label{bad-1}
For $n< k$, we have
$$I^{1}_{k,n}\lesssim2^{n-k}\lambda^{2}\|p_{n}\|^{2}_{2}.$$
\end{lem}
\begin{proof}

First, the cancellation property-Proposition~\ref{czdecom}(iii) of $b_{n}$ gives
$$I^{1}_{k,n}=\sum_{i\in \mathcal{S}_{k}}\sum_{x\in\Z}
\Big\|\frac{1}{|A_{n_{i+1}}|}\sum_{y\in \mathcal{I}(x+A_{n_{i+1}}\setminus A_{n_{i}},n)}b_{n}(y)\Big\|^{2}_{L_{2}(\mathcal{M})}.$$
On the other hand, observe that
\begin{align}\label{bad34237289}
	b_{n}=p_n f q_n+q_{n+1}fp_n-q_{n+1}f_np_n.
\end{align}
Indeed, by (\ref{czd1}), $p_{n}=q_{n+1}-q_{n}\leq q_{n+1}$; moreover, by Cuculescu's construction-Lemma \ref{cucu}(i), we obtain
$$p_n f_{n} q_n=p_nq_{n+1} f_{n} q_{n+1}q_n=p_nq_nq_{n+1} f_{n} q_{n+1}=0.$$
This gives the desired expression (\ref{bad34237289}).

With the observation (\ref{bad34237289}) and the Minkowski inequality, we see that to obtain the desired inequality for the term $I^{1}_{k,n}$, it suffices to estimate the following three terms
\begin{eqnarray*}
	I^{1}_{k,n,1}& \triangleq &\sum_{i\in \mathcal{S}_{k}}\sum_{x\in\Z}
	\Big\|\frac{1}{|A_{n_{i+1}}|}\sum_{y\in \mathcal{I}(x+A_{n_{i+1}}\setminus A_{n_{i}},n)}(p_n f q_n)(y)\Big\|^{2}_{L_{2}(\mathcal{M})}\\
	I^{1}_{k,n,2}& \triangleq &\sum_{i\in \mathcal{S}_{k}}\sum_{x\in\Z}
	\Big\|\frac{1}{|A_{n_{i+1}}|}\sum_{y\in \mathcal{I}(x+A_{n_{i+1}}\setminus A_{n_{i}},n)}(q_{n+1} f p_n)(y)\Big\|^{2}_{L_{2}(\mathcal{M})}\\
	I^{1}_{k,n,3}& \triangleq &\sum_{i\in \mathcal{S}_{k}}\sum_{x\in\Z}
	\Big\|\frac{1}{|A_{n_{i+1}}|}\sum_{y\in \mathcal{I}(x+A_{n_{i+1}}\setminus A_{n_{i}},n)}(q_{n+1}f_np_n)(y)\Big\|^{2}_{L_{2}(\mathcal{M})},
\end{eqnarray*}
We first deal with the term $I^{1}_{k,n,1}$. By Lemma~\ref{KeyLem}, we  have
\begin{equation}\label{Bad-1}
\begin{split}
&\Big\|\frac{1}{|A_{n_{i+1}}|}\sum_{y\in\mathcal{I}(x+A_{n_{i+1}}\setminus A_{n_{i}},n)}(p_{n}fq_n)(y)\Big\|^{2}_{L_{2}(\mathcal{M})}\\
&\lesssim 2^{n-2k}\sum_{y\in \mathcal{I}_1(x+A_{n_{i+1}}\setminus A_{n_{i}},n)}\tau(|\mathsf{E}_n(p_{n}fq_n\1_{x+A_{n_{i+1}}\setminus A_{n_{i}}})(y)|^2),
\end{split}
\end{equation}
where we used the fact that for all $i\in\mathcal S_k$, $2^k\le |A_{n_{i}}|\le2^{k+1}$ and $|\mathcal{I}_1(x+A_{n_{i+1}}\setminus A_{n_{i}},n)|\lesssim 2^{n}$.

Note that $f\1_{x+A_{n_{i+1}}\setminus A_{n_{i}}}$ is positive in $\mathcal{N}$ and $f\1_{x+A_{n_{i+1}}\setminus A_{n_{i}}}\le f$. Since $\mathsf{E}_n$ is a positive map, we obtain $\mathsf{E}_n(f\1_{x+A_{n_{i+1}}\setminus A_{n_{i}}})\le f_n$. Moreover, by the H\"{o}lder inequality,
\begin{equation}\label{bad-s1}
\begin{split}
	&\quad\tau(|p_n(y)\mathsf{E}_n(f\1_{x+A_{n_{i+1}}\setminus A_{n_{i}}})(y)q_n(y)|^2)\\
	& =\tau\Big(p_{n}(y)\mathsf{E}_n(f\1_{x+A_{n_{i+1}}\setminus A_{n_{i}}})(y)q_{n}(y)\mathsf{E}_n(f\1_{x+A_{n_{i+1}}\setminus A_{n_{i}}})(y)p_{n}(y)\Big)\\
	& \leq\tau\Big(p_{n}(y)\mathsf{E}_n(f\1_{x+A_{n_{i+1}}\setminus A_{n_{i}}})(y)p_{n}(y)\Big)\|q_{n}(y)\mathsf{E}_n(f\1_{x+A_{n_{i+1}}\setminus A_{n_{i}}})(y)q_{n}(y)\|_{\M}\\
	& \leq\tau\Big(p_{n}(y)\mathsf{E}_n(f\1_{x+A_{n_{i+1}}\setminus A_{n_{i}}})(y)p_{n}(y)\Big)\|q_{n}(y)f_{n}(y)q_{n}(y)\|_{\M}\\
	& \leq\lambda\tau\Big(p_{n}(y)\mathsf{E}_n(f\1_{x+A_{n_{i+1}}\setminus A_{n_{i}}})(y)p_{n}(y)\Big).
\end{split}
\end{equation}
Combining the above estimate with~\eqref{Bad-1}, we get
\begin{align*}
  I^{1}_{k,n,1}\lesssim\lambda2^{n-2k}\sum_{x\in\Z}\sum_{i\in \mathcal{S}_{k}}\sum_{y\in\mathcal{I}_1(x+A_{n_{i+1}}\setminus A_{n_{i}},n)}\tau(p_{n}(y)\mathsf{E}_n(f\1_{x+A_{n_{i+1}}\setminus A_{n_{i}}})(y)p_{n}(y)).
\end{align*}
Note that for any fixed $x\in\Z$, $\cup_{i\in\mathcal S_k}\mathcal{I}(x+A_{n_{i+1}}\setminus A_{n_{i}},n)\subseteq x+A_{2^{k+1}}$, which implies
\begin{equation*}
\begin{split}
&\sum_{i\in \mathcal{S}_{k}}\sum_{y\in\mathcal{I}_1(x+A_{n_{i+1}}\setminus A_{n_{i}},n)}\tau(p_{n}(y)\mathsf{E}_n(f\1_{x+A_{n_{i+1}}\setminus A_{n_{i}}})(y)p_{n}(y))\\
  &=\sum_{i\in\mathcal S_k}\sum_{\substack{I\in\I_n,\\ I\cap\partial(x+A_{n_{i+1}}\setminus A_{n_{i}})\neq\emptyset}}\sum_{y\in I}\tau(\mathsf{E}_n(p_nf\1_{x+A_{n_{i+1}}\setminus A_{n_{i}}}p_n)(y))\\
  &=\sum_{i\in\mathcal S_k}\sum_{\substack{I\in\I_n,\\ I\cap\partial(x+A_{n_{i+1}}\setminus A_{n_{i}})\neq\emptyset}}\sum_{y\in I\cap (x+A_{n_{i+1}}\setminus A_{n_{i}})}\tau(p_n(y)f(y)p_n(y))\\
  &=\sum_{i\in\mathcal S_k}\sum_{y\in\mathcal{I}(x+A_{n_{i+1}}\setminus A_{n_{i}},n)}\tau(p_n(y)f(y)p_n(y))\\
  &\le \sum_{y\in x+A_{2^{k+1}}}\tau(p_n(y)f(y)p_n(y)),\\
\end{split}
\end{equation*}
where in the second equality, {we used the fact that $I\in\sigma_n$, so $\chi_I\otimes\1_\M\in \mathcal N_n$ and $\varphi\circ\mathsf{E}_n=\varphi$.} Therefore, putting these observations together, we finally get
\begin{align*}
   I^{1}_{k,n,1}&\lesssim\lambda2^{n-2k}\sum_{x\in\Z}\sum_{y\in x+A_{2^{k+1}}}\tau(p_n(y)f(y)p_n(y))\lesssim\lambda2^{n-k}\sum_{x\in\Z}\tau(p_n(x)f(x)p_n(x))\\
   &\lesssim\lambda^22^{n-k}\sum_{x\in\Z}\tau(p_n(x))=\lambda^22^{n-k}\|p_n\|_2^2,
\end{align*}
where the second inequality followed from the Fubini theorem and the last inequality from the trace-preserving property of $\mathsf{E}_n$  and (\ref{czd1123}). This finishes the proof of $I^{1}_{k,n,1}$.

We now turn to the terms
$I^{1}_{k,n,2}$ and $I^{1}_{k,n,3}$. Note that $q_{n+1}\in\mathcal N_n$ and
\begin{align}\label{bad3423729}
q_{n+1}f_nq_{n+1}\lesssim q_{n+1}f_{n+1}q_{n+1}\leq\lambda.
\end{align}
Then the argument \eqref{bad-s1} also works for the terms $I^{1}_{k,n,2}$ and $I^{1}_{k,n,3}$. As a consequence, we can estimate these two terms in a similar way as in the proof of $I^{1}_{k,n,1}$. Hence, we omit the details and the proof is complete.
\end{proof}

\medskip

\subsubsection{{\bf{Estimate of}} $I^{2}_{k,n}$}
\begin{lem}\label{bad12333}\rm
	For $n< k$, we have
	$$I^{2}_{k,n}\lesssim2^{n-k}\lambda^{2}\|p_{n}\|^{2}_{2}.$$
\end{lem}
\begin{proof}
By the same arguments of $I^{1}_{k,n}$, that is using the cancellation property and the definition of $b_n$, we are reduced to estimating the following three terms:
\begin{eqnarray*}
	I^{2}_{k,n,1}& \triangleq &\sum_{i\in \mathcal{S}_{k}}\Big(\frac{1}{|A_{n_{i}}|}-\frac{1}{|A_{n_{i+1}}|}\Big)^{2}\sum_{x\in\Z}\Big\|\sum_{y\in \mathcal{I}(x+A_{n_{i}},n)}(p_n f q_n)(y)\Big\|^{2}_{L_{2}(\mathcal{M})}\\
	I^{2}_{k,n,2}& \triangleq &\sum_{i\in \mathcal{S}_{k}}\Big(\frac{1}{|A_{n_{i}}|}-\frac{1}{|A_{n_{i+1}}|}\Big)^{2}\sum_{x\in\Z}\Big\|\sum_{y\in \mathcal{I}(x+A_{n_{i}},n)}(q_{n+1} f p_n)(y)\Big\|^{2}_{L_{2}(\mathcal{M})}\\
	I^{2}_{k,n,3}& \triangleq &\sum_{i\in \mathcal{S}_{k}}\Big(\frac{1}{|A_{n_{i}}|}-\frac{1}{|A_{n_{i+1}}|}\Big)^{2}\sum_{x\in\Z}\Big\|\sum_{y\in \mathcal{I}(x+A_{n_{i}},n)}(q_{n+1} f_n p_n)(y)\Big\|^{2}_{L_{2}(\mathcal{M})}.
\end{eqnarray*}

We begin with the term $I^{2}_{k,n,1}$. Using Lemma~\ref{KeyLem}, the fact that $\mathcal{I}_1(x+A_{n_{i}},n)\subseteq x+A_{2^{k+2}}$ and $|\mathcal{I}_{1}(x+A_{n_{i}},n)|\lesssim2^{n}$ for all ${i\in\mathcal{S}_k}$, we deduce that
\begin{align*}
  \Big\|\sum_{y\in\mathcal{I}(x+A_{n_{i}},n)}(p_{n}fq_n)(y)\Big\|^{2}_{L_{2}(\M)}&\lesssim 2^{n}\sum_{y\in \mathcal{I}_1(x+A_{n_{i}},n)}\tau(|\mathsf{E}_n(p_{n}fq_n\1_{x+A_{n_{i}}})(y)|^2)\\
  &\le 2^{n}\sum_{y\in x+A_{2^{k+2}}}\tau(|\mathsf{E}_n(p_{n}fq_n\1_{x+A_{n_{i}}})(y)|^2).
\end{align*}
Since $\mathsf{E}_n$ is a unital completely positive map, 
by H\"{o}lder's inequality and (\ref{czd1123}), we obtain
\begin{align*}
\tau(|\mathsf{E}_n(p_{n}fq_n\1_{x+A_{n_{i}}})(y)|^2)\leq\tau\Big(p_{n}(y)f_{n}(y)p_{n}(y)\Big)\|q_{n}(y)f_{n}(y)q_{n}(y)\|_{\M}\lesssim\lambda^{2}\tau(p_{n}(y)).
\end{align*}
Combining the above observations with the fact that for $i\in\mathcal{S}_{k}$, $A_{2^{k}}\subset A_{n_{i}}\subset A_{2^{k+1}}$, we have
\begin{align*}
  I^{2}_{k,n,1}&\lesssim2^{n}\lambda^2\sum_{i\in \mathcal{S}_{k}}\Big(\frac{1}{|A_{n_{i}}|}-\frac{1}{|A_{n_{i+1}}|}\Big)^{2}\sum_{x\in\Z}\sum_{y\in x+A_{2^{k+2}}}\tau(p_{n}(y))\\
  &\le 2^{n}\lambda^2\Big(\frac{1}{|A_{2^{k}}|}-\frac{1}{|A_{2^{k+1}}|}\Big)^2 \sum_{x\in\Z}\sum_{z\in x+A_{n_{2^{k+2}}}} \tau(p_{n}(z))\\
  &\lesssim 2^{n-k}\lambda^2\|p_n\|_{2}^2.
\end{align*}

The same arguments also work  for the terms $I^{2}_{k,n,2}$ and $I^{2}_{k,n,3}$  just by noticing the relation (\ref{bad3423729}) and $q_{n+1}\in\mathcal N_n$, we omit the proofs. The lemma is proved.
\end{proof}

\bigskip

\begin{proof}[Proof of estimate \eqref{bad1}.]
By Lemmas \ref{badfunction1}, \ref{bad-1} and \ref{bad12333},  we conclude the desired estimate \eqref{bad1} and complete the argument for $Tb$.
\end{proof}
\subsection{Weak type estimate for the good function:~(\ref{HLX15})}\quad

In order to estimate the good part, we need the following proposition, which is a Hilbert-valued analogue of the commutative result. Its proof could be done quite similarly as in the classical setting (see e.g.~\cite{JRW98}) and we omit it.
\begin{prop}\label{t1}
Let $h\in L_{2}(\mathcal N)$. Then
$$\|Th\|_{L_{2}(L_{\infty}(\Omega)\overline{\otimes}\mathcal N)}\lesssim\|h\|_{2}.$$
\end{prop}

With this proposition at hand, we prove easily the estimate \eqref{HLX15}.
\begin{proof}[Proof of estimate \eqref{HLX15}.]
We clearly have
$$\widetilde{\varphi}(\chi_{(\lambda/2,\infty)}(|Tg|))\leq\frac{\|Tg\|^{2}_{L_{2}(L_{\infty}(\Omega)\overline{\otimes}\mathcal N)}}{\lambda^{2}}
\lesssim\frac{\|g\|_{2}^{2}}{\lambda^{2}}\leq\frac{\|g\|_{1}\|g\|_{\infty}}{\lambda^{2}}
\lesssim\frac{\|f\|_{1}}{\lambda},$$
as a consequence of the Chebychev inequality, Proposition~\ref{t1}, the H\"{o}lder inequality and conclusion (ii) in Theorem \ref{czdecom}.
This completes the proof.
\end{proof}

\section{Proof of Theorem \ref{t5343232}: $(L_{\infty},\mathrm{BMO})$ and strong type $(p,p)$ estimates}\label{strong}
In this section, we examine the $(L_{\infty},\mathrm{BMO})$ and strong type $(p,p)$ estimates stated in Theorem \ref{t5343232}.

\subsection{$(L_{\infty},\mathrm{BMO})$ estimate}\quad

\medskip

We first recall the definition of $\mathrm{BMO}$ spaces associated
to the von Neumann algebra $\mathcal R=\mathcal N
\overline{\otimes} \mathcal{B}(\ell_{2})$ equipped with the tensor trace $\psi=\varphi\otimes tr$ where $tr$ is the canonical trace on $\mathcal{B}(\ell_{2})$.
According to \cite{MP}, the dyadic $\mathrm{BMO}$ space $\mathrm{BMO}_{d}(\mathcal R)$ is defined as a subspace of $L_{\infty}(\M\overline{\otimes} \mathcal{B}(\ell_{2});L_2^{rc}(\Z;dx/(1+|x|)^{2}))$ with
$$\|f\|_{\mathrm{BMO}_{d}(\mathcal R)} \, = \, \max \Big\{
\|f\|_{\mathrm{BMO}_{\! d}^r(\mathcal R)}, \|f\|_{\mathrm{BMO}_{\! d}^c(\mathcal R)}
\Big\} \, < \, \infty,$$ where the row and column dyadic $\mathrm{BMO}_{d}$
norms are given by
\begin{eqnarray*}
\|f\|_{\mathrm{BMO}_{\! d}^r(\mathcal R)} & = & \sup_{I \in\I} \Big\| \Big(
\frac{1}{|I|} \sum_{x\in I} \Big|\big(f(x) -\frac1{|I|}\sum_{y\in I} f(y)\big)^{\ast}\Big|^2 \Big)^{\frac12} \Big\|_{\M\overline{\otimes} \mathcal{B}(\ell_{2})}, \\
\|f\|_{\mathrm{BMO}_{\! d}^c(\mathcal R)} & = & \sup_{I \in\I} \Big\| \Big(
\frac{1}{|I|} \sum_{x\in I} \Big|f(x) -\frac1{|I|}\sum_{y\in I} f(y)\Big|^2  \Big)^{\frac12} \Big\|_{\M\overline{\otimes} \mathcal{B}(\ell_{2})}.
\end{eqnarray*}

By the definition of BMO spaces, we are reduced to showing
\begin{align}\label{94}
\Big\|
\sum_{i\in \mathcal{S}} T_i \hskip-1pt f \otimes e_{i1}
\Big\|_{\mathrm{BMO}_{d}(\mathcal{R})} \lesssim \,
\|f\|_\infty,
\end{align}
and
\begin{align}\label{924}
	\Big\|
	\sum_{i\in \mathcal{S}} T_i \hskip-1pt f \otimes e_{1i}
	\Big\|_{\mathrm{BMO}_{d}(\mathcal{R})} \lesssim \,
	\|f\|_\infty.
\end{align}
However, it suffices to estimate (\ref{94}) since~\eqref{924} can be obtained just by noting $(T_if)^*=T_i(f^*)$.  Notice that (\ref{94}) is equivalent to
\begin{align}\label{95}
\Big\|
\sum_{i\in \mathcal{S}} T_i \hskip-1pt f \otimes e_{i1}
\Big\|_{\mathrm{BMO}_{d}^{c}(\mathcal R)} \lesssim \,
\|f\|_\infty
\end{align}
and
\begin{align}\label{96}
\Big\|
\sum_{i\in \mathcal{S}} T_i \hskip-1pt f \otimes e_{i1}
\Big\|_{\mathrm{BMO}_{d}^{r}(\mathcal R)} \lesssim \,
\|f\|_\infty.
\end{align}

Now let us prove (\ref{94}).
\begin{proof}[Proof of (\ref{94}).]
Let $f\in L_{\infty}(\mathcal N)$ and $I$ be a dyadic cube in $\Z$. Decompose $f$ as
 $f=f\chi_{ 3I}+f\chi_{\Z\setminus 3I}\triangleq f_1+f_2$, where $3I$ denotes the interval with the same center as $I$ such that $|3I|=3|I|$. If we set $\alpha_{I,i}=T_{i}f_2(c_I)$ where $c_I$ is the center of $I$ or the left neighborhood in $\mathbb Z$ of the center if the center does not belong to $\mathbb Z$, and $\alpha_{I}=\sum\limits_{i}\alpha_{I,i}\otimes e_{i1}$, then
$$T_{i}f(x)-\alpha_{I,i}=T_{i}f_1(x)+(T_{i}f_2(x)-\alpha_{I,i})\triangleq B_{i1}f+B_{i2}f.$$
We first prove (\ref{95}). By the operator convexity of square function $x\mapsto|x|^{2}$, we obtain
$$\big|\sum_{i\in \mathcal{S}}(T_{i}f-\alpha_{I,i})\otimes e_{i1}\big|^{2}\leq2\big|\sum_{i\in \mathcal{S}}B_{i1}f\otimes e_{i1}\big|^{2}+2\big|\sum_{i\in \mathcal{S}}B_{i2}f\otimes e_{i1}\big|^{2}.$$
The first term $B_{1}f=\sum\limits_{i} B_{i1} \hskip-1pt f \otimes e_{i1}$ is easy to estimate. Indeed,
\begin{align*}
  \Big\| \Big(\frac{1}{|I|}\sum_{x\in I} (B_1
f(x))^\ast (B_1f(x)) \Big)^{\frac{1}{2}} \Big\|_{\M \overline{\otimes}
\mathcal{B}(\ell_{2})}^{2}&=\Big\| \Big(\frac{1}{|I|}\sum_{x\in I} \sum_{i\in \mathcal{S}}|T_{i}f_{1}(x)|^{2} \Big)^{\frac{1}{2}} \Big\|^{2}_{\M}\\
& =\frac{1}{|I|}\Big\| \sum_{x\in I} \sum_{i\in \mathcal{S}}|T_{i}f_{1}(x)|^{2} \Big\|_{\M}\\
& \le \frac{1}{|I|} \sup_{\|a\|_{L_2(\M)} \le 1}\tau \sum_{x\in\Z}
\sum_{i\in \mathcal{S}}|T_{i}f_{1}(x)a|^{2}\\
& = \frac{1}{|I|} \sup_{\|a\|_{L_2(\M)} \le 1} \|(T_{i}(f
\chi_{3I}a))_{i\in\mathcal{S}}\|_{L_2(\mathcal{N};\ell_{2}^{rc})}^2\\
& \lesssim \|f\|_\infty^{2},
\end{align*}
where in the third inequality, we considered elements in $\M$ as bounded linear operators on $L_{2}(\M)$ via the left multiplication and the last inequality follows from the $L_2$-boundedness of $T$, namely Proposition \ref{t1}.

Now we turn to the second term $B_{2}f=\sum\limits_{i\in \mathcal{S}} B_{i2} \hskip-1pt f \otimes e_{i1}$. Note that
\begin{align*}
&B_{2}f(x)^{\ast}B_{2}f(x)=\sum_{i\in \mathcal{S}}|T_{i}f_{2}(x)-T_{i}f_{2}(c_{I})|^{2}\\
&\ =\sum_{i\in \mathcal{S}}|(M_{A_{n_{i}}}f_{2}(x)-M_{A_{n_{i+1}}}f_{2}(x))-
(M_{A_{n_{i}}}f_{2}(c_{I})-M_{A_{n_{i+1}}}f_{2}(c_{I}))|^2\\
&\ =\sum_{k}\sum_{i\in \mathcal{S}_{k}}|(M_{A_{n_{i}}}f_{2}(x)-M_{A_{n_{i+1}}}f_{2}(x))-
(M_{A_{n_{i}}}f_{2}(c_{I})-M_{A_{n_{i+1}}}f_{2}(c_{I}))|^2\\
&\ \triangleq\sum_{k}\sum_{i\in \mathcal{S}_{k}}|F_{k,i}(x)|^{2}.
\end{align*}
We claim that for any $k$ satisfying $2^{k+1}<|I|$, $F_{k,i}(x)=0$ for any $i\in \mathcal{S}_{k}$ and $x\in I$. Indeed, fix $i\in \mathcal{S}_{k}$ and $x\in I$. Since $2^{k+1}<|I|$ and $f_2$ is supported in $\Z\setminus 3I$, a simple geometric observation implies that both $M_{A_{n_{i+1}}}f_{2}$ and $M_{A_{n_{i}}}f_{2}$ are supported in $\Z\setminus I$. This is precisely the claim.  Hence,
$$B_{2}f(x)^{\ast}B_{2}f(x)=\sum_{k:2^{k+1} \geq |I|}\sum_{i\in \mathcal{S}_{k}}|F_{k,i}(x)|^{2}.$$
In the following, we further split the summation over $\mathcal{S}_{k}$ into two parts by comparing $n_{i+1}-n_i$ and $|I|$. More precisely,
we decompose $B_{2}f(x)^{\ast}B_{2}f(x)$ as
\begin{align*}
  \sum_{k:2^{k+1} \geq |I|}\sum_{{\begin{subarray}{c}
i\in \mathcal{S}_{k} \\ n_{i+1}-n_{i}<|I|
\end{subarray}}}|F_{k,i}(x)|^{2}+\sum_{k:2^{k+1} \geq|I|}\sum_{{\begin{subarray}{c}
i\in \mathcal{S}_{k} \\ n_{i+1}-n_{i}\geq|I|
\end{subarray}}}|F_{k,i}(x)|^{2}.
\end{align*}

Let us estimate terms of the case $n_{i+1}-n_{i}<|I|$. By the operator-convexity of square function $x\mapsto|x|^{2}$,
\begin{align*}
	\begin{split}
	  \sum_{k:2^{k+1} \geq|I|}\sum_{{\begin{subarray}{c}
				i\in \mathcal{S}_{k} \\ n_{i+1}-n_{i}<|I|
	\end{subarray}}}|F_{k,i}(x)|^{2}&\lesssim\sum_{k:2^{k+1} \geq|I|}\sum_{{\begin{subarray}{c}
	i\in \mathcal{S}_{k} \\ n_{i+1}-n_{i}<|I|
\end{subarray}}}|M_{A_{n_{i}}}f_{2}(x)-M_{A_{n_{i+1}}}f_{2}(x)|^2\\
		&\quad+\sum_{k:2^{k+1} \geq|I|}\sum_{{\begin{subarray}{c}
					i\in \mathcal{S}_{k} \\ n_{i+1}-n_{i}<|I|
		\end{subarray}}}|M_{A_{n_{i}}}f_{2}(c_{I})-M_{A_{n_{i+1}}}f_{2}(c_{I})|^2.
	\end{split}
\end{align*}
Now we claim that to complete the argument of the case $n_{i+1}-n_{i}<|I|$, it is enough to show for any $z\in I$
\begin{align}\label{235}
\|M_{A_{n_{i}}}f_{2}(z)-M_{A_{n_{i+1}}}f_{2}(z)\|_{\M}\lesssim\|f\|_{\infty}|I|^{\frac{1}{2}}\Big(
\int_{|A_{n_{i}}|}^{|A_{n_{i+1}}|}\frac{1}{u^{2}}du\Big)^{\frac{1}{2}}.
\end{align}
Indeed, by (\ref{235}) and the Minkowski inequality, we have
\begin{align*}
\Big\|\Big(\frac{1}{|I|}\sum_{x\in I}\sum_{k:2^{k+1} \geq|I|}\sum_{{\begin{subarray}{c}
				i\in \mathcal{S}_{k} \\ n_{i+1}-n_{i}<|I|
	\end{subarray}}}|F_{k,i}(x)|^{2}\Big)^{\frac{1}{2}}\Big\|^2_{\M}
&=\Big\|\frac{1}{|I|}\sum_{k:2^{k+1} \geq|I|}
\sum_{{\begin{subarray}{c}
			i\in \mathcal{S}_{k} \\ n_{i+1}-n_{i}<|I|
\end{subarray}}}|F_{k,i}(x)|^{2}\Big\|_{\M} \\
&\leq\frac{1}{|I|}\sum_{k:2^{k+1} \geq|I|}\sum_{{\begin{subarray}{c}
				i\in \mathcal{S}_{k} \\ n_{i+1}-n_{i}<|I|
	\end{subarray}}}\|F_{k,i}(x)\|_{\M}^{2}\\
	&\lesssim\|f\|^2_{\infty}\sum_{k:2^{k+1} \geq|I|}\sum_{i\in \mathcal{S}_{k} }\int_{|A_{n_{i}}|}^{|A_{n_{i+1}}|}\frac{1}{u^{2}}du\\
		& \leq\|f\|^2_{\infty}\sum_{k:2^{k+1} \geq|I|}\int_{2^{k}}^{2^{k+1}}\frac{1}{u^{2}}du\\
	& \lesssim\|f\|^2_{\infty},
\end{align*}
which is the desired estimate. It remains to show (\ref{235}). To this end, fixing $z\in I$, and applying the H\"{o}lder inequality and the fact that $n_{i+1}-n_{i}<|I|$, we have
\begin{align*}
&	\|M_{A_{n_{i}}}f_{2}(z)-M_{A_{n_{i+1}}}f_{2}(z)\|_{\M}\\
&=\Big\|\Big(\frac{1}{|A_{n_{i}}|}-\frac{1}{|A_{n_{i+1}}|}\Big)\sum_{y\in z+A_{n_{i}}}f_{2}(y)+\frac{1}{|A_{n_{i+1}}|}
\sum_{y\in z+A_{n_{i+1}}\setminus A_{n_{i}}}f_{2}(y)\Big\|_{\M}\\
&\leq\Big(\frac{1}{|A_{n_{i}}|}-\frac{1}{|A_{n_{i+1}}|}\Big)|A_{n_{i}}|\|f_{2}\|_{\infty}+\frac{1}{|A_{n_{i+1}}|}
(|A_{n_{i+1}}|-|A_{n_{i}}|)\|f_{2}\|_{\infty}\\
&\lesssim\|f\|_{\infty}\int_{|A_{n_{i}}|}^{|A_{n_{i+1}}|}\frac{1}{u}du\\
& \lesssim\|f\|_{\infty}|I|^{\frac{1}{2}}\Big(
\int_{|A_{n_{i}}|}^{|A_{n_{i+1}}|}\frac{1}{u^{2}}du\Big)^{\frac{1}{2}}.
\end{align*}

Let us now turn to the case $n_{i+1}-n_{i}\geq |I|$. Likewise, we use the operator-convexity of square function $x\mapsto|x|^{2}$ to obtain
\begin{align*}
	\begin{split}
	\sum_{k:2^{k+1} \geq|I|}\sum_{{\begin{subarray}{c}
				i\in \mathcal{S}_{k} \\ n_{i+1}-n_{i}\geq |I|
	\end{subarray}}}|F_{k,i}(x)|^{2}&\lesssim\sum_{k:2^{k+1} \geq|I|}\sum_{{\begin{subarray}{c}
	i\in \mathcal{S}_{k} \\ n_{i+1}-n_{i}\geq |I|
\end{subarray}}}|M_{A_{n_{i}}}f_{2}(x)-M_{A_{n_{i}}}f_{2}(c_{I})|^2\\
		&\quad+\sum_{k:2^{k+1} \geq|I|}\sum_{{\begin{subarray}{c}
					i\in \mathcal{S}_{k} \\ n_{i+1}-n_{i}\geq|I|
		\end{subarray}}}|M_{A_{n_{i+1}}}f_{2}(x)-M_{A_{n_{i+1}}}f_{2}(c_{I})|^2.
	\end{split}
\end{align*}
In this case, observe that for any $A_{n_{i}}$,
\begin{align*}
	\|M_{A_{n_{i}}}f_{2}(x)-M_{A_{n_{i}}}f_{2}(c_{I})\|_{\M}
	& =\frac{1}{|A_{n_{i}}|}\sum_{y\in\Z}\|f_{2}(y)(\chi_{x+A_{n_{i}}}(y)-\chi_{c_I+A_{n_{i}}}(y))\|_\infty\\
		&\leq \frac{1}{|A_{n_{i}}|}\|f\|_{\infty}|(c_{I}+A_{n_{i}})\Delta (x+A_{n_{i}})|,
\end{align*}
where $\Delta$ denotes the usual symmetric difference of two sets. Note that $|(c_{I}+A_{n_{i}})\Delta (x+A_{n_{i}})|\lesssim |x-c_I|$. Then
$$\|M_{A_{n_{i}}}f_{2}(x)-M_{A_{n_{i}}}f_{2}(c_{I})\|_{\infty}\lesssim\frac{1}{|A_{n_{i}}|}|x-c_I|\|f\|_{\infty}.$$
Moreover, it is not difficult to verify that the number of $i\in \mathcal{S}_{k}$ such that $n_{i+1}-n_i\geq|I|$  is smaller than $\frac{2^k}{|I|}$. Therefore,
\begin{align*}
\Big\|\Big(\frac{1}{|I|}\sum_{x\in I}\sum_{k:2^{k+1} \geq|I|}\sum_{{\begin{subarray}{c}
				i\in \mathcal{S}_{k} \\ n_{i+1}-n_{i}\geq|I|
	\end{subarray}}}|F_{k,i}(x)|^{2}\Big)^{\frac{1}{2}}\Big\|^2_{\infty}&\leq\frac{1}{|I|}\sum_{x\in I}\sum_{k:2^{k+1} \geq|I|}\sum_{{\begin{subarray}{c}
				i\in \mathcal{S}_{k} \\ n_{i+1}-n_{i}\geq|I|
	\end{subarray}}}\|F_{k,i}(x)\|^{2}_{\infty}\\
	& \lesssim\|f\|^{2}_{\infty}\frac{1}{|I|}\sum_{x\in I}\sum_{k:2^{k+1} \geq|I|}\frac{2^k}{|I|}\frac{1}{|A_{n_{i}}|^{2}}|x-c_I|^2\\
	& \leq\|f\|^2_{\infty}\sum_{k:2^{k+1} \geq|I|}\frac{2^k}{|I|}\frac{1}{|A_{n_{i}}|^{2}}|I|^2\\
	&\leq\|f\|^2_{\infty}|I|\sum_{k:2^{k+1} \geq|I|}2^{-k-1}\lesssim\|f\|^2_{\infty},
\end{align*}
where we used the relations $|x-c_I|\leq|I|$ and $|A_{n_i}|\thickapprox2^{k+1}$. So we complete the proof of (\ref{95}).

We now consider (\ref{96}). However, we just deal with the term $B_{1}f=\sum\limits_{i\in \mathcal{S}} B_{i1} \hskip-1pt f \otimes e_{i1}$, since $B_{2}f$ can be treated as before.
We note that
\begin{align*}
& \Big\| \Big(\frac{1}{|I|}\sum_{x\in I} (B_1
f(x)) (B_1f(x))^\ast\Big)^{\frac{1}{2}} \Big\|_{\M  \overline{\otimes}
\mathcal{B}(\ell_{2})}^{2}\\
& = \frac{1}{|I|} \Big\|
\sum_{i_1,i_2\in \mathcal{S}} \Big[ \sum_{x\in I} T_{i_1}f_{1}(x)
T_{i_2} f_{1}^\ast(x) \Big] \otimes e_{i_1,i_2} \Big\|_{\M
 \overline{\otimes} \mathcal{B}(\ell_{2})} \\
&\triangleq \frac{1}{|I|}
\|\Lambda\|_{\M  \overline{\otimes} \mathcal{B}(\ell_{2})}.
\end{align*}
Note that $\Lambda$ is a positive operator acting on $\ell_{2}(L_2(\M))$ $(=L_2(\M;\ell_{2}^{rc}))$. Hence,
\begin{align*}
 &\frac{1}{|I|} \Big\| \sum_{x\in I} (B_1 \hskip-1pt
f(x))(B_1 \hskip-1pt f(x))^\ast \Big\|_{\M  \overline{\otimes}
\mathcal{B}(\ell_{2})} = \frac{1}{|I|} \sup_{\|a\|_{L_2(\M;\ell_{2}^{rc})} \le 1} \big\langle \Lambda a, a \big\rangle \\
& =  \frac{1}{|I|} \sup_{\|a\|_{L_2(\M;\ell_{2}^{rc})} \le 1} \hskip4.7pt
\tau \Big( \big[ \sum_{i_1\in \mathcal{S}} a_{i_1}^\ast \otimes e_{1i_1}
\big] \, \Lambda \,
\big[ \sum_{i_2\in \mathcal{S}} a_{i_2} \otimes e_{i_21} \big] \Big) \\
& = \frac{1}{|I|} \sup_{\|a\|_{L_2(\M;\ell_{2}^{rc})} \le 1}
\hskip4.7pt \tau \sum_{x\in I} \Big| \sum_{i\in \mathcal{S}}
T_i f_{1}^\ast(x) a_i \Big|^2\\
& \lesssim  \frac{1}{|I|} \sup_{\|a\|_{L_2(\M;\ell_{2}^{rc})} \le 1}
\hskip4.7pt \tau \sum_{x\in\Z}\sum_{i\in \mathcal{S}}
|f_{1}^\ast (x)a_i|^2  \\
& \le  \frac{1}{|I|}
\sup_{\|a\|_{L_2(\M;\ell_{2}^{rc})} \le 1}
\hskip4.7pt \tau (\sum_{i:i\in \mathcal{S}}|a_{i}|^2) |3I| \, \|f\|_\infty^2 \ \lesssim
\ \|f\|_\infty^2,
\end{align*}
where in the first inequality we applied Proposition \ref{t1}. This proves (\ref{96}).
Finally, putting all the estimates obtained so far together with their row analogues, we get
$$\max \Big\{ \Big\|
\sum_{i\in \mathcal{S}} T_i \hskip-1pt f \otimes e_{i1}
\Big\|_{\mathrm{BMO}_{d}(\mathcal{R})}, \Big\|
\sum_{i\in \mathcal{S}} T_i \hskip-1pt f \otimes e_{1i}
\Big\|_{\mathrm{BMO}_{d}(\mathcal{R})} \Big\} \lesssim \|f\|_\infty.$$
This completes the proof of the $(L_{\infty},\mathrm{BMO})$ estimate.
\end{proof}
\subsection{Strong type $(p,p)$ estimates}\quad

\medskip

In this subsection, we complete the proof of Theorem \ref{t5343232} by showing the strong type $(p,p)$ estimates.
\begin{prop}\label{11}\rm
Let $1<p<\infty$. Then $(T_i)_{i\in \mathcal{S}}$ is bounded from $L_{p}(\mathcal{N})$ to $L_p(\mathcal{N};
\ell_{2}^{rc})$.
\end{prop}

\begin{proof}
Proposition \ref{t1} gives the result for $p=2$. For the case $1<p<2$, by
applying the weak type $(1,1)$ estimate of ${T}$ and Proposition \ref{t1}, we conclude that ${T}$ is bounded from $L_{p}(\mathcal{N})$ to $L_p(L_{\infty}(\Omega)\overline{\otimes}\mathcal{N})$ by real interpolation \cite{P2}. Thus $(T_i)_{i\in \mathcal{S}}$ is bounded from $L_{p}(\mathcal{N})$ to $L_p(\mathcal{N};
\ell_{2}^{rc})$ according to noncommutative Khintchine's inequalities---Proposition~\ref{nonkin}. 

Consider the case $2<p<\infty$.
If we set $T_{c}f=\sum_{i\in \mathcal{S}} T_i \hskip-1pt f \otimes e_{i1}$ and $T_{r}f=\sum_{i\in \mathcal{S}} T_i \hskip-1pt f \otimes e_{1i}$, then Proposition~\ref{t1} and $(L_\infty,\mathrm{BMO})$ estimates yield that $T_{c}$ and $T_{r}$ are bounded from $L_p(\mathcal N)$ to $L_p(\mathcal{N}\overline{\otimes}\mathcal B({\ell_2}))$ via complex interpolation \cite{Mu}. Therefore, $(T_i)$ is bounded from $L_p(\mathcal N)$ to $L_p(\mathcal{N};
\ell_{2}^{rc})$ for all $2< p<\infty$.
\end{proof}

\section{Proof of Theorem~\ref{theorem11}}\label{main}
In this section, we prove Proposition \ref{trans} and Theorem~\ref{theorem11}.
\begin{proof}[Proof of Proposition~\ref{trans}]
In the following we only consider the one-sided ergodic averages while the two-sided ones can be handled quite similarly.
By the standard approximation argument stated in Remark \ref{app}, it suffices to
show for any fixed integer $i_0\geq1$
\begin{align}\label{erg1}
\big\|\big(M_{n_{i}}(T)x-M_{n_{i+1}}(T)x\big)_{0\leq i\leq i_0}\big\|_{L_p(\M;\ell^{rc}_{2})} \lesssim\|x\|_{L_{p}(\M)}.
\end{align}

For each $n\in\N$, define
$$M^{\prime}_{n}:L_{p}(\Z;L_{p}(\M))\rightarrow L_{p}(\Z;L_{p}(\M)),\ M^{\prime}_{n}f(k)=\frac{1}{n+1}\sum_{l=0}^{n}f(l+k),\ \forall k\in\Z.$$
Set $N=n_{i_0+1}$. Let $m$ be a large integer bigger than $N$.  Fix $x\in L_{p}(\M)$. Define a $L_{p}(\M)$-valued function $f_{m}$ on $\Z$ as
$$f_{m}(l)=T^{l}x,\ \ \mbox{if}\ 0\le l\leq m+N;\ \ f_{m}(l)=0,\, \mbox{otherwise}.$$
Then for all $0\leq k\leq m$ and $1\leq n\leq N$,
$$T^{k}M_{n}(T)x=\frac{1}{n+1}\sum_{l=0}^{n}T^{k+l}x=\frac{1}{n+1}\sum_{l=0}^{n}f_{m}(l+k)=M^{\prime}_{n}f_{m}(k).$$
Therefore, for $0\leq k\leq m$ and $0\le i\le i_0$,
$$T^{k}(M_{n_{i}}(T)x-M_{n_{i+1}}(T)x)=M^{\prime}_{n_{i}}f_{m}(k)-M^{\prime}_{n_{i+1}}f_{m}(k).$$
Note that $T$ is a power bounded operator, namely $\sup_{k\in\Z}\|T^k\|_{L_p(\M)\rightarrow L_p(\M)}<\8$. Then for each $k\in\mathbb N$, Lemma~\ref{exten-lin} implies
\begin{align}\label{ergo32}
\begin{split}
  &\big\|\big(M_{n_{i}}(T)x-M_{n_{i+1}}(T)x\big)_{0\leq i\leq i_0}\big\|_{L_p(\M;\ell^{rc}_{2})}\\
& \lesssim\big\|\big(T^{k}(M_{n_{i}}(T)x-M_{n_{i+1}}(T)x)\big)_{0\leq i\leq i_0}\big\|_{L_p(\M;\ell^{rc}_{2})} \\
& =\big\|\big(M^{\prime}_{n_{i}}f_{m}(k)-M^{\prime}_{n_{i+1}}f_{m}(k)\big)_{0\leq i\leq i_0}\big\|_{L_p(\M;\ell^{rc}_{2})}.
\end{split}
\end{align}

Now we prove (\ref{erg1}). Note that by the Fubini theorem and the noncommutative Khintchine inequalities, namely Proposition~\ref{nonkin}, it is clear that $L_p(\M\otimes\ell_\8(\mathbb N);\ell_2^{rc})\approx\ell_p(\mathbb N;L_p(\M;\ell_2^{rc}))$. Based on this observation,~\eqref{ergo32} and the assumption~\eqref{assume-1}, we have
\begin{align*}
&\big\|\big(M_{n_{i}}(T)x-M_{n_{i+1}}(T)x\big)_{0\leq i\leq i_0}\big\|^p_{L_p(\M;\ell^{rc}_{2})}\\ &\lesssim\frac{1}{m+1}\sum_{k=0}^{m}\big\|\big(M^{\prime}_{n_{i}}f_{m}(k)-M^{\prime}_{n_{i+1}}f_{m}(k)\big)_{0\leq i\leq i_0}\big\|^{p}_{L_p(\M;\ell^{rc}_{2})}\\
  &\lesssim\frac{1}{m+1}\big\|\big(M^{\prime}_{n_{i}}f_{m}-M^{\prime}_{n_{i+1}}f_{m}\big)_{0\leq i\leq i_0}\big\|^p_{L_p(\M\overline{\otimes}\ell_{\infty}(\mathbb Z);\ell^{rc}_{2})}\\
  &\lesssim\frac{1}{m+1}\|f_{m}\|^{p}_{L_p(\M\overline{\otimes}\ell_{\infty}(\Z))}\lesssim\frac{m+N+1}{m+1}\|x\|^{p}_{L_{p}(\M)}.
\end{align*}
By the arbitrariness of $m$, we find
$$\big\|\big(M_{n_{i}}(T)x-M_{n_{i+1}}(T)x\big)_{0\leq i\leq i_0}\big\|_{L_p(\M;\ell^{rc}_{2})}\lesssim\|x\|_{L_{p}(\M)}.$$
So we finish the proof of (\ref{erg1}).

\end{proof}

Now, we are able to conclude Theorem~\ref{theorem11}.
\begin{proof}[Proof of Theorem~\ref{theorem11}.]
Actually, Theorem~\ref{theorem11} follows immediately from Proposition~\ref{trans} and Theorem~\ref{Th5}.
\end{proof}

\section{Proof of Theorem \ref{thm-1}}\label{main-thm2}

In this section, we prove Proposition~\ref{extend of Lam} and Theorem~\ref{thm-1}. Before that, we
give some notations and lemmas. We first recall the definitions of {\it{dilations}} and {\it{$N$-dilation}} in the noncommutative setting (see e.g. \cite{HRW}).
\begin{defi}
Let $1\le p\le\infty$ and $T:L_p(\mathcal{M},\tau_{\mathcal M})\rightarrow L_p(\mathcal{M},\tau_{\mathcal M})$ be a contraction. We call the operator $T$ has a {\it{dilation}} if there exist a von Neumann algebra $\mathcal A$ equipped with a normal faithful semifinite trace $\tau_{\mathcal A}$, two contraction linear operators $Q:L_p(\mathcal{A},\tau_{\mathcal A})\rightarrow L_p(\mathcal{M},\tau_{\mathcal M})$, $J:L_p(\mathcal{M},\tau_{\mathcal M})\rightarrow L_p(\mathcal{A},\tau_{\mathcal A})$ and an isometry  $U:L_p(\mathcal{A},\tau_{\mathcal A})\rightarrow L_p(\mathcal{A},\tau_{\mathcal A})$ such that 
	\begin{equation}\label{decom of T}
		T^n=QU^nJ,\ \forall n\in\mathbb{N}\cup\{0\}.
	\end{equation}
\end{defi}
We can use the following diagrams to represent the above decomposition
\begin{displaymath}
	\xymatrix{ L_p(\mathcal M, \tau_{\mathcal M}) \ar[r]^{T^n} \ar[d]_J & L_p(\mathcal M, \tau_{\mathcal M}) \\
		L_p(\mathcal A, \tau_{\mathcal A}) \ar[r]^{U^n}                  & L_p(\mathcal A, \tau_{\mathcal A})\ar[u]_Q}
\end{displaymath}
for all $n\ge 0$.

We call the operator $T$ has an {\it{$N$-dilation}} if~\eqref{decom of T} holds for all $n\in\{0,1,\cdots, N\}$. 
	
Let $1\le p<\8$.  Let $\mathbb{SS}(L_p(\mathcal M))$ denote the set of all Lamperti contractions on $L_p(\mathcal M)$. Moreover, for a given  set $A$ consisting of operators on $L_p(\mathcal M)$, we denote by $\mbox{conv}\{A\}$ the convex hull of $A$, namely
$$\mbox{conv}\{A\}=\bigg\{\sum_{i=1}^n\lambda_iT_i:T_i\in A,\sum_{i=1}^n\lambda_i=1, \lambda_i\ge 0, n\in\mathbb{N}\bigg\}.$$

The following result which can be seen as a {\it{$N$-dilation}} theorem for $\mbox{conv}(\mathbb{SS}(L_p(\mathcal M)))$ was established in~\cite[Corollary 4.7]{HRW}.
\begin{lem}\label{dilation}
Let $1< p<\8$. Each operator $T\in \emph{conv}(\mathbb{SS}(L_p(\mathcal M)))$ has an $N$-dilation for all $N\in\mathbb{N}$.
\end{lem}

Now we present a characterization theorem for isometric operators established in \cite{Yea81,Jun-Ruan-Sher05}.
Let $\mathcal M$ and $\mathcal A$ be two von Neumann algebras. A complex linear map $J:\mathcal M\rightarrow\mathcal A$ is called a Jordan $*$-homomorphism if $J(x)^*=J(x^*)$ and $J(x^2)=J(x)^2$ for all $x\in\mathcal M$. Moreover, $J$ is called the Jordan $*$-monomorphism if $J$ is an injective Jordan $*$-homomorphism.

The following lemma was obtained by St{\o}rmer~\cite{St65}.
\begin{lem}\label{Jordan}
Let $J:\mathcal M\rightarrow\mathcal A$ be a normal (completely additive, ultraweakly continuous) Jordan $*$-homomorphism. Let $\widetilde{\mathcal{A}}$ denote the von Neumann subalgebra generated by $J(\mathcal{M})$ in $\mathcal{A}$, and $\mathcal{Z}_{\widetilde{\mathcal{A}}}$ be the center of $\widetilde{\mathcal{A}}$. Then there are two projections $e,f\in \mathcal{Z}_{\widetilde{\mathcal{A}}}$ satisfying $e+f=1_{\widetilde{\mathcal{A}}}$, such that the map $x\rightarrow J(x)e$ is a $*$-homomorphism and $x\rightarrow J(x)f$ is a $*$-anti-homomorphism.
\end{lem}
The following characterization of isometric operators will be frequently used.
\begin{prop}\cite{Yea81,Jun-Ruan-Sher05}\label{isometry}
Let $1\le p\neq 2<\8$. Assume that  $T:L_p(\mathcal{M},\tau_{\mathcal M})\rightarrow L_p(\mathcal{A},\tau_{\mathcal A})$ is a bounded linear operator. Then $T$ is
an isometry if and only if there exist uniquely a partial isometry $w\in\mathcal{A}$, a normal Jordan $*$-monomorphism $J:\mathcal M\rightarrow\mathcal A$, and a positive self-adjoint operator $b$ affiliated with $\mathcal{A}$, such that
\begin{enumerate}[\noindent]
  \item~\emph{(i)}~$w^{*}w=supp b=J(1)$;
  \item~\emph{(ii)}~For all $x\in\mathcal M$, $J(x)$ commutes with every spectral projection of $b$;
  \item~\emph{(iii)}~$T(x)=wbJ(x)$ for all $x\in\mathcal{S}_{\mathcal M}$;
  \item~\emph{(iv)}~$\tau_{\mathcal{A}}(b^pJ(x))=\tau_{\mathcal{M}}(x)$ for all $x\in\mathcal{M}_+$.
\end{enumerate}
\end{prop}
Now we are ready to prove Proposition~\ref{extend of Lam}.
\begin{proof}[Proof of Proposition~\ref{extend of Lam}]
Let $(x_n)_{1\leq n\leq N}$ be a finite sequence in $L_p(\mathcal M;\ell_2^{rc})$. By an approximation argument, without loss of generality, we may assume that $x_n\in\mathcal{S}_{\M}$ for all $1\le n\le N$. For $p=2$, the conclusion is trivially right. Indeed, in this case
\begin{equation*}
  \big\|\big(T(x_n)\big)_{n}\big\|^2_{L_2(\mathcal M;\ell_2^{rc})}=\sum_{n}\|Tx_n\|^2_{L_2(\M)}=\|(x_n)_{n}\|_{L_2(\M;\ell_2^{rc})}.
\end{equation*}

We now focus on the case $2<p<\8$. 
Since $T$ is an isometric operator, there exist $w,b,J$ such that $T=wbJ$ satisfying properties (i)-(iii) in Proposition~\ref{isometry}. Moreover, by Lemma~\ref{Jordan}, we may find projections $e,f$  with the property that $e$ and $f$ commute with $b$ such that
\begin{equation}\label{isom-T}
\begin{split}
  \big(T(x_n)\big)^*\big(T(x_n)\big)&=\big|\big(bJ(x_n)e+bJ(x_n)f\big)\big|^2=b^2J(x^*_nx_n)e+b^2J(x_nx^*_n)f.
\end{split}
\end{equation}

As a consequence,
	\begin{equation*}
		\begin{split}
			\sum_{n}\big(T(x_n)\big)^*\big(T(x_n)\big)&=b^2J\big(\sum_{n}x^*_nx_n\big)e+b^2J\big(\sum_{n}x_nx^*_n\big)f.
		\end{split}
	\end{equation*}
Set $y_1=\big(\sum_{n}x^*_nx_n\big)^{1/2}$ and $y_2=\big(\sum_{n}x_nx^*_n\big)^{1/2}$. Then
\begin{equation*}
 \sum_{n}\big(T(x_n)\big)^*\big(T(x_n)\big)=b^2J(y_1^2)e+b^2J(y_2^2)f.
\end{equation*}
Moreover, applying the property that $e$ and $f$ commute with $b$ again, we have
\begin{equation}\label{12}
	\begin{split}
		\big\|\big(T(x_n)\big)_{n}\big\|^p_{L_p(\mathcal M;\ell_2^c)}&=\tau\Big(\big(b^2J(y_1^2)e+b^2J(y_2^2)f\big)^{\frac p2}\Big)\\
&=\tau(b^pJ(y_1^p)e)+\tau(b^pJ(y_2^p)f).
	\end{split}
\end{equation}

On the other hand, by the fact that $T=wbJ$, we get
\begin{align*}
   &\big(T(x_n)\big)\big(T(x_n)\big)^*=(wbJ(x_n))(wbJ(x_n))^*\\
   &=wb^2J(x_n)J(x_n)^*w^*=wb^2J(x_nx_n^*)ew^*+wb^2J(x_n^*x_n)fw^*,
\end{align*}
which gives rise to
\begin{equation*}
\sum_{n} \big(T(x_n)\big)\big(T(x_n)\big)^*=wb^2J(y_2^2)ew^*+wb^2J(y_1^2)fw^*.
\end{equation*}
Then by the above observations and Proposition \ref{isometry}(i), we find
\begin{equation}\label{isom-T2}
\begin{split}
   &\big\|\big(T(x_n)\big)_{n}\big\|^p_{L_p(\mathcal M;\ell_2^r)}=\|w(b^2J(y_2^2)e+b^2J(y_1^2)f)w^*\|_{L_{\frac p2}(\M)}^{\frac p2}\\
   &=\|(b^2J(y_2^2)e+b^2J(y_1^2)f)^{\frac12}w^*w(b^2J(y_2^2)e+b^2J(y_1^2)f)^{\frac12}\|_{L_{\frac p2}(\M)}^{\frac p2}\\
   &=\|b^2J(y_2^2)e+b^2J(y_1^2)f\|_{L_{\frac p2}(\M)}^{\frac p2}\\
   &=\tau(b^pJ(y_2^p)e)+\tau(b^pJ(y_1^p)f).
\end{split}
\end{equation}
Therefore, combining (\ref{12}) with (\ref{isom-T2}), we arrive at
\begin{align*}
  &\big\|\big(T(x_n)\big)_{n}\big\|^p_{L_p(\mathcal M;\ell_2^c)}+\big\|\big(T(x_n)\big)_{n}\big\|^p_{L_p(\mathcal M;\ell_2^r)}\\
  &=\tau(b^pJ(y_1^p)e)+\tau(b^pJ(y_2^p)f)+\tau(b^pJ(y_2^p)e)+\tau(b^pJ(y_1^p)f)\\
  &=\tau(b^pJ(y_1^p))+\tau(b^pJ(y_2^p)).
\end{align*}
Moreover, by Proposition~\ref{isometry}(iv), we have
\begin{equation*}
\begin{split}
  \big\|\big(T(x_n)\big)_{n}\big\|^p_{L_p(\mathcal M;\ell_2^c)}+\big\|\big(T(x_n)\big)_{n}\big\|^p_{L_p(\mathcal M;\ell_2^r)}&=\tau(y_1^p)+\tau(y_2^p)\\
  &=\|(x_n)_n\|^p_{L_p(\mathcal M;\ell_2^r)}+\|(x_n)_n\|^p_{L_p(\mathcal M;\ell_2^c)},
\end{split}
\end{equation*}
which yields  $\|\big(T(x_n)\big)_{n}\|_{L_p(\M;\ell_2^{rc})}=\|(x_n)_n\|_{L_p(\M;\ell_2^{rc})}$. So  $T$ is an isometry on $L_p(\M;\ell_2^{rc})$ and we finish the proof of the case $2<p<\infty$.

We then turn to the case $1< p<2$. Since $(x_n)_{n}\in L_p(\mathcal M;\ell_2^{rc})$, for any given $\varepsilon>0$, there are two finite sequences $(g_n)_{n}$ and $(h_n)_{n}$ such that $x_n=g_n+h_n$ and
	\begin{equation}\label{Mix-norm}
		\|(g_n)_{n}\|^p_{L_p(\mathcal M;\ell_2^{c})}+\|(h_n)_{n}\|^p_{L_p(\mathcal M;\ell_2^{r})}\le \|(x_n)_{n}\|^p_{L_p(\mathcal M;\ell_2^{rc})}+\varepsilon.
	\end{equation}
Then by Lemma \ref{Jordan}, we may decompose $T(x_n)$ as
\begin{equation*}
  T(x_n)=T(g_n)e+T(h_n)f+(T(g_n)f+T(h_n)e).
\end{equation*}
By the definition of the norm $\|\cdot\|_{L_p(\mathcal M;\ell_2^{rc})}$, we have
\begin{equation}\label{p<2}
	\begin{split}
		&\big\|\big(T(x_n)\big)_{n}\big\|^p_{L_p(\mathcal M;\ell_2^{rc})}\\
&\le \big\|\big(T(g_n)e+T(h_n)f\big)_{n}\big\|^p_{L_p(\mathcal M;\ell_2^{c})}+
\big\|\big(T(g_n)f+T(h_n)e\big)_{n}\big\|^p_{L_p(\mathcal M;\ell_2^{r})}.
\end{split}
\end{equation}

Set $y_3=\big(\sum_{n}g^*_ng_n\big)^{1/2}$ and $y_4=\big(\sum_{n}h_nh^*_n\big)^{1/2}$.  Similar to (\ref{12}), one has
\begin{equation}\label{isom-T4}
   \big\|\big(T(g_n)e+T(h_n)f\big)_{n}\big\|^p_{L_p(\mathcal M;\ell_2^{c})}=\tau(b^pJ(y_3^p)e)+\tau(b^pJ(y_4^p)f);
\end{equation}
and similar to (\ref{isom-T2}), one gets
\begin{equation}\label{isom-T5}
  \big\|\big(T(g_n)f+T(h_n)e\big)_{n}\big\|^p_{L_p(\mathcal M;\ell_2^{r})}=\tau(b^pJ(y_3^p)f)+\tau(b^pJ(y_4^p)e).
\end{equation}
Together~\eqref{isom-T4},~\eqref{isom-T5} and~\eqref{p<2} with Proposition \eqref{isometry}(iv) and~\eqref{Mix-norm}, we get
\begin{equation*}
  \big\|\big(T(x_n)\big)_{n}\big\|_{L_p(\mathcal M;\ell_2^{rc})}^p\le \|(x_n)_{n}\|^p_{L_p(\mathcal M;\ell_2^{rc})}+\varepsilon.
\end{equation*}
By the arbitrariness of $\varepsilon$, we proved that $T$ extends to a contraction on $L_p(\M;\ell_2^{rc})$.

Finally, we assume that $T$ is positive. It is sufficient to consider the case $1<p<2$. 
The following facts are taken from \cite{HRW}.
Note that $T$ is a positive isometry. Then $T$ is Lamperti and $T=bJ$, where $b,J$ are defined in Lemma \ref{Jordan} (see \cite[Theorem 3.3 and Remark 3.9]{HRW}). Let $\mathcal{A}$ be the von Neumann algebra generated by $J(\mathcal M)$. By Lemma~\ref{Jordan}, we may decompose $\mathcal{A}=\mathcal{A}_1\oplus \mathcal{A}_2$, where $\mathcal{A}_1$ and $\mathcal{A}_2$ are two von Neumann subalgebras of $\mathcal{A}$, and write $J=J_1+J_2$ such that $J_1:\M\rightarrow \mathcal{A}_1$ is a normal $*$-homomorphism and $J_2:\M\rightarrow \mathcal{A}_2$ is a normal $*$-anti-homomorphism. Let $\sigma: \mathcal{A}_2\rightarrow\mathcal{A}_2^{op}$ be the usual opposite map and define
$$\Sigma:\mathcal{A}\rightarrow\mathcal{A}_1\oplus\mathcal{A}_2^{op},~\quad \Sigma=\mbox{Id}_{\mathcal{A}_1}\oplus\sigma.$$
Then $\Sigma\circ J$ is a normal $*$-homomorphism and $\Sigma(J(\M))$ is a von Neumann subalgebra of $\mathcal{A}_1\oplus\mathcal{A}_2^{op}$.
Define
$$\varphi:\Sigma(J(\M))_+\rightarrow[0,\8],\quad x\mapsto\tau(b^p\Sigma^{-1}x).$$
Then $\varphi$ becomes a normal semifinite trace on $\Sigma(J(\M))$. Let $L_p(\Sigma(J(\M)),\varphi)$ be the associated noncommutative $L_p$ space. Then $\Sigma\circ J$ extends to a positive surjective isometry
$$\tilde{J}:L_p(\M,\tau)\rightarrow L_p(\Sigma(J(\M)),\varphi),\quad x\mapsto\Sigma(J(x)),$$
whence $\tilde{J}^{-1}$ is well defined, positive and isometric on $L_p(\Sigma(J(\M)),\varphi)$. We refer to \cite[Proposition 5.3]{HRW} for the  detailed description of above facts.

As being an extension of the normal $*$-homomorphism $\Sigma\circ J$, it is easy to check that $\tilde{J}$ extends to an isometry from $L_p(\M,\tau;\ell_2^{rc})$ onto $L_p(\Sigma(J(\M)),\varphi;\ell_2^{rc})$. And thus $\tilde{J}^{-1}$ extends also to an isometry from $L_p(\Sigma(J(\M)),\varphi;\ell_2^{rc})$ to $L_p(\M,\tau;\ell_2^{rc})$.

To complete the argument,
there remains to verify that the embedding
$$L_p(\Sigma(J(\M)),\varphi;\ell_2^{rc})\rightarrow L_p(\M,\tau;\ell_2^{rc}),\quad (x_n)_n\mapsto(b\Sigma^{-1}x_n)_n$$
is an isometry. Let $p^\prime$ be the conjugate index of $p$, that is $\frac{1}{p}+\frac{1}{p^\prime}=1$. Then $2<p^\prime<\8$. Let $(x_n)$ be a finite sequence in $L_{p^\prime}(\Sigma(J(M)),\varphi;\ell_2^{rc})$. Write $x_n=(x_{n,1},x_{n,2})\in L_{p^\prime}(\mathcal A_1)\oplus L_{p^\prime}(\mathcal A_2^{op})$. Then
\begin{align*}
&\|(b^{p/p^\prime}\Sigma^{-1}x_n)_n\|^{p^\prime}_{L_{p^\prime}(\M,\tau;\ell_2^{c})}=\Big\|\Big(\sum_n|b^{p/p^\prime}\Sigma^{-1}x_n|^2\Big)^{1/2}
\Big\|^{p^\prime}_{L_{p^\prime}(\M,\tau)}\\
&=\tau\Big(b^{p}\Big(\sum_n\Sigma^{-1}(x_n^*)\Sigma^{-1}(x_n)\Big)^{p^\prime/2}\Big)=
\tau\Big(b^{p}\Big(\sum_n\Big(x_{n,1}^*x_{n,1}+\sigma^{-1}(x_{n,2}x^*_{n,2})\Big)\Big)^{p^\prime/2}\Big)\\
&=\tau\Big(b^{p}\Big(\sum_nx_{n,1}^*x_{n,1}+\sigma^{-1}\big(\sum_nx_{n,2}x^*_{n,2}\big)\Big)^{p^\prime/2}\Big)\\
&=\tau\Big(b^{p}\Big(\sum_nx_{n,1}^*x_{n,1}\Big)^{p^\prime/2}\Big)+\tau\Big(b^p\Big(\sigma^{-1}\big(\sum_nx_{n,2}x^*_{n,2}\big)\Big)^{p^\prime/2}\Big)\\
&=\|(x_{n,1})_n\|^{p^\prime}_{L_{p^\prime}(\Sigma(J(M)),\varphi;\ell_2^{c})}+\|(x_{n,2})_n\|^{p^\prime}_{L_{p^\prime}(\Sigma(J(M)),\varphi;\ell_2^{r})}.
\end{align*}
Note that $(b^{p/p^\prime}\Sigma^{-1}x_n)^*=b^{p/p^\prime}\Sigma^{-1}x_n^*$. Then replacing $b^{p/p^\prime}\Sigma^{-1}x_n$ by $(b^{p/p^\prime}\Sigma^{-1}x_n)^*$ in the above identities, we deduce that
$$\|(b^{p/p^\prime}\Sigma^{-1}x_n)_n\|^{p^\prime}_{L_{p^\prime}(\M,\tau;\ell_2^{r})}=
\|(x_{n,1})_n\|^{p^\prime}_{L_{p^\prime}(\Sigma(J(M)),\varphi;\ell_2^{r})}+\|(x_{n,2})_n\|^{p^\prime}_{L_{p^\prime}(\Sigma(J(M)),\varphi;\ell_2^{c})}.$$
Combining the above two identities, we have
\begin{align*}
  &\|(b^{p/p^\prime}\Sigma^{-1}x_n)_n\|^{p^\prime}_{L_{p^\prime}(\M,\tau;\ell_2^{c})}
  +\|(b^{p/p^\prime}\Sigma^{-1}x_n)_n\|^{p^\prime}_{L_{p^\prime}(\M,\tau;\ell_2^{r})}\\
  &=\|(x_{n,1})_n\|^{p^\prime}_{L_{p^\prime}(\Sigma(J(M)),\varphi;\ell_2^{r})}+\|(x_{n,2})_n\|^{p^\prime}_{L_{p^\prime}(\Sigma(J(M)),\varphi;\ell_2^{r})}\\
  &+\|(x_{n,1})_n\|^{p^\prime}_{L_{p^\prime}(\Sigma(J(M)),\varphi;\ell_2^{c})}+\|(x_{n,2})_n\|^{p^\prime}_{L_{p^\prime}(\Sigma(J(M)),\varphi;\ell_2^{c})}\\
  &=\|(x_n)_n\|^{p^\prime}_{L_{p^\prime}(\Sigma(J(M)),\varphi;\ell_2^{rc})}.
\end{align*}
Hence the map
$$\phi:L_{p^\prime}(\Sigma(J(M)),\varphi;\ell_2^{rc})\rightarrow L_{p^\prime}(\M,\tau;\ell_2^{rc}),\quad (x_n)_n\mapsto(b^{p/p^\prime}\Sigma^{-1}x_n)_n$$
is an isometry.

By Proposition~\ref{dual}, we know that $(L_p(\Sigma(J(M)),\varphi;\ell_2^{rc}))^*=L_{p^\prime}(\Sigma(J(M)),\varphi;\ell_2^{rc})$. Let $\phi^*$ be the conjugate of $\phi$. Then we can check that $  \phi^*((b\Sigma^{-1}x_n)_n)=(x_n)_n$ (see the proof of Proposition 5.3 in~\cite{HRW}).

Therefore, recalling that $T$ is a contraction on $L_p(\M,\tau;\ell_2^{rc})$ with $T=bJ$ and together with above observations, we finally get
\begin{align*}
  &\|(x_n)_n\|_{L_{p}(\Sigma(J(M)),\varphi;\ell_2^{rc})}=\|\phi^*((b\Sigma^{-1}x_n)_n)\|_{L_{p}(\Sigma(J(M)),\varphi;\ell_2^{rc})}\\
  &\le \|(b\Sigma^{-1}x_n)_n\|_{L_{p}(\M,\tau;\ell_2^{rc})}=\|(T\tilde{J}^{-1}x_n)_n\|_{L_{p}(\M,\tau;\ell_2^{rc})}\\
  &\le\|(\tilde{J}^{-1}x_n)_n\|_{L_{p}(\M,\tau;\ell_2^{rc})}=\|(x_n)_n\|_{L_{p}(\Sigma(J(M)),\varphi;\ell_2^{rc})},
\end{align*}
which gives $\|(x_n)_n\|_{L_{p}(\Sigma(J(M)),\varphi;\ell_2^{rc})}=\|(b\Sigma^{-1}x_n)_n\|_{L_{p}(\M,\tau;\ell_2^{rc})}$. So the proof of Proposition~\ref{extend of Lam} is complete.
\end{proof}

\begin{remark}{\rm
For a positive isometry $T$ on  $L_p(\M)$ ($1< p<2$), once one gets that it extends to a contraction on $L_p(\M;\ell_2^{rc})$, there is another way to conclude its isometry: For $T=bJ$, we consider another operator $T'(x)=b^{p-1}J(x)$ initially defined on $\mathcal S_\mathcal M$, which extends to an isometry on $L_{p'}(\mathcal M;\ell_2^{rc})$. Thus one may deduce that
$$\big|\sum_n\tau(b^pJ(x_n)J(y_n))\big|\leq \big\|(T(x_n))_n\|_{{L_{p}(\M,\tau;\ell_2^{rc})}}\|(y_n)_n\|_{L_{p'}(\M,\tau;\ell_2^{rc})}.$$
 But by commutation relations and the tracial property $$\tau(b^pJ(x_n)J(y_n)) = \tau(b^pJ(y_n)J(x_n)),$$ and thus $$2\tau(b^pJ(x_n)J(y_n)) =\tau(b^pJ(x_ny_n+y_nx_n)) = 2\tau(x_ny_n).$$ Therefore,
 $$\big|\sum_n\tau(x_ny_n)\big|\leq \big\|(T(x_n))_n\|_{{L_{p}(\M,\tau;\ell_2^{rc})}}\|(y_n)_n\|_{L_{p'}(\M,\tau;\ell_2^{rc})},$$
 which gives
 $$\big\|(x_n)_n\|_{{L_{p}(\M,\tau;\ell_2^{rc})}}\leq\big\|(T(x_n))_n\|_{{L_{p}(\M,\tau;\ell_2^{rc})}}.$$

 Unfortunately, at the moment of writing, we still have no idea on how to remove the positivity assumption in the case $1< p<2$.
}
\end{remark}


Based on Proposition~\ref{extend of Lam}, a similar Coifman-Weiss's transference principle as Proposition \ref{trans} holds also for isometries and we thus get the following quantitative mean ergodic theorem for isometries.
\begin{lem}\label{thm-iso}
Let $1< p<\infty$ and $T:L_p(\mathcal M)\rightarrow L_p(\mathcal M)$ be an isometry. Let $(n_{i})_{i\in\N}$ be any increasing sequence of positive integers. Then for $2\le p<\infty$,
$$\big\|\big(M_{n_{i}}(T)x-M_{n_{i+1}}(T)x\big)_{i\in \N}\big\|_{L_p(\M;\ell^{rc}_{2})}\lesssim\|x\|_{L_{p}(\M)},\forall x\in L_p(\M).$$
If $T$ is moreover positive, then the above inequality holds also for $1<p<2$.
\end{lem}

Lemma \ref{thm-iso} can be proved in a similar way as Proposition \ref{trans} by observing that the ``$\lesssim$" in \eqref{ergo32} can be strengthened to ``$=$" due to Proposition~\ref{extend of Lam}. We omit the details.

By Lemma \ref{thm-iso}, we now can conclude the proof of Theorem \ref{thm-1}.
\begin{proof}[Proof of Theorem \ref{thm-1}]
Recall that $M_n(T)=\frac{1}{n+1}\sum_{k=0}^nT^k$. By the density argument, it is enough to prove the theorem for any finite positive sequence $(n_i)_{0\le i\le N}$.

Fix an arbitrary $N\geq1$. Note that $T\in \mathfrak{S}$. Then by definition, we can find $(T_j)\subseteq\mbox{conv}(\mathbb{SS}(L_p(\mathcal M)))$ such that $T_j$ converges to $T$ in the sense of strong operator topology. Moreover, according to Lemma~\ref{dilation}, for each $T_j$, there exist two contractions $Q_{j,N}$, $J_{j,N}$ and one isometry $U_{j,N}$ such that $T_j^{n_i}=Q_{j,N}U_{j,N}^{n_i}J_{j,N}$ for every $0\le n_i\le n_N$. As a consequence, by Proposition~\ref{exten-lin}, $Q_{j,N}$ and $J_{j,N}$ extend to two bounded operators on $L_p(\mathcal M;\ell_2^{rc})$.
Therefore, for any fixed $x\in L_p(\mathcal M)$, together with  Lemma~\ref{thm-iso}, we get
\begin{align*}
&\|\big(M_{n_{i}}(T_j)x-M_{n_{i+1}}(T_j)x\big)_{0\le i\le N}\|_{L_p(\mathcal M;\ell_2^{rc})}\\&=\|(Q_{j,N}(M_{n_i}-M_{n_{i+1}})(U_{j,N})J_{j,N}x)_{0\le i\le N}\|_{L_p(\mathcal M;\ell_2^{rc})}\\
&\le \|((M_{n_i}-M_{n_{i+1}})(U_{j,N})J_{j,N}x)_{0\le i\le N}\|_{L_p(\mathcal M;\ell_2^{rc})}\\
&\lesssim\|J_{j,N}x\|_{L_p(\mathcal M)}\lesssim\|x\|_{L_p(\mathcal M)}.
\end{align*}

Consequently, combined with the noncommutative Khintchine inequality in $L_p$ space-Proposition~\ref{nonkin},
we finally deduce that
\begin{align*}
&\|\big(M_{n_{i}}(T)x-M_{n_{i+1}}(T)x\big)_{0\le i\le N}\|_{L_p(\mathcal M;\ell_2^{rc})}\\
&\le \|\big(\big(M_{n_{i}}-M_{n_{i+1}}\big)T_jx)_{0\le i\le N}\|_{L_p(\mathcal M;\ell_2^{rc})}+\|\big(\big(M_{n_{i}}-M_{n_{i+1}}\big)(T-T_j)x\big))_{0\le i\le N}\|_{L_p(\mathcal M;\ell_2^{rc})}\\
&\lesssim\|x\|_{L_p(\mathcal M)}+\|\big(\big(M_{n_{i}}-M_{n_{i+1}}\big)(T-T_j)x\big))_{0\le i\le N}\|_{L_p(\mathcal M;\ell_2^{rc})}\\
&\approx\|x\|_{L_p(\mathcal M)}+\bigg\|\sum_{i=0}^N\varepsilon_{i}\big(M_{n_{i}}-M_{n_{i+1}}\big)(T-T_j)x\bigg\|_{L_p(\Omega;L_p(\mathcal M))}\\
&\lesssim\|x\|_{L_p(\mathcal M)}+\sum_{i=0}^N\|\big(M_{n_{i}}-M_{n_{i+1}}\big)(T-T_j)x\|_{L_p(\mathcal M)},
\end{align*}
By letting $j\rightarrow\infty$, we obtain the desired conclusion.
\end{proof}

\begin{remark}\rm
Let $1<p<\8$ and let $(\Omega,\mu)$ be a $\sigma$-finite measure space. Suppose that $T:L_p(\Omega)\rightarrow L_p(\Omega)$ is a positive contraction. Define $\widetilde{T}=T\otimes I_{L_p(\M)}$.   By the Akcoglu dilation theorem (cf.~\cite{A}) and Theorem~\ref{thm-1}, we immediately obtain the quantitative mean ergodic theorem for $\widetilde{T}$, namely
$$\big\|\big(M_{n_{i}}(\widetilde{T})x-M_{n_{i+1}}(\widetilde{T})x\big)_{i\in \N}\big\|_{L_p(L_\infty(\Omega)\overline{\otimes}\M;\ell^{rc}_{2})}\lesssim \|x\|_{p}, ~\forall\ x\in L_p(L_\infty(\Omega)\overline{\otimes}\M;\ell_2^{rc}),$$
where $(n_{i})_{i\in\N}$ is any increasing sequence of positive integers.

\end{remark}



\begin{thebibliography}{1}
\bibitem{A}M. A. Akcoglu, A pointwise ergodic theorem in $L_{p}$-spaces, Canad. J. Math. \textbf{27} (1975), no. 5, 1075-1082.
\bibitem{An}C. Anantharaman-Delaroche, On ergodic theorems for free group actions on noncommutative spaces, Probab. Theory Related Fields \textbf{135} (2006), no. 4, 520-546.
\bibitem{AR}J. Avigad, J. Rute, Oscillation and the mean ergodic theorem for uniformly convex Banach spaces, Ergod. Th. \& Dynam. Sys. \textbf{35} (2015), no. 4, 1009-1027.
\bibitem{Be} T. N. Bekjan, Noncommutative maximal ergodic theorems for positive contractions, J. Funct. Anal. \textbf{254} (2008), no. 9, 2401-2418.
\bibitem{Bour89} J. Bourgain, Pointwise ergodic theorems for arithmetic sets, Publ. Math. IHES. \textbf{69} (1989), 5-41.
\bibitem{C1}L. Cadilhac, Noncommutative Khintchine inequalities in interpolation spaces of $L_{p}$-spaces, Adv. Math. \textbf{352} (2019) 265-296.
\bibitem{CCP}L. Cadilhac, J. M. Conde-Alonso, J. Parcet, Spectral multipliers in group algebras and noncommutative Calder\'on-Zygmund theory, J. Math. Pures Appl. (9) \textbf{163} (2022) 450-472.
\bibitem{CoWe} R.R. Coifman, W. Weiss, Transference methods in Analysis, CBMS No. 31, Amer. Math. Soc., Providence, R.I., 1997.
\bibitem{CoDa78} J. P. Conze, N. Dang-Ngoc, Ergodic theorems for noncommutative dynamical systems, Invent. Math., \textbf{46} (1978), 1-15.
\bibitem{Cuc}I. Cuculescu, Martingales on von Neumann algebras, J. Multivariate Anal. \textbf{1} (1971), no. 1, 17-27.


\bibitem{DS58} N. Dunford, J. T. Schwartz, Linear Operators. I. General Theory, Applied Mathematics, Vol. 7. Interscience Publishers, Inc., New York, 1958.
\bibitem{FK}T. Fack, H. Kosaki, Generalized $s$-numbers of $\tau$-measurable operators, Pacific J. Math. \textbf{123} (1986), no. 2, 269-300.
\bibitem{Hong}G. Hong, The behaviour of square functions from ergodic theory in $L^{\infty}$, Proc. Amer. Math. Soc. \textbf{143} (2015), no. 11, 4797-4802.
\bibitem{HLW}G. Hong, B. Liao, S. Wang, Noncommutative maximal ergodic inequalities associated with
doubling conditions, Duke Math. J. \textbf{170} (2021), no. 2, 205-246.
\bibitem{HLX}G. Hong, W. Liu, B. Xu, Quantitative mean ergodic inequalities II: power bounded operators acting on one noncommutative $L_p$ space, prepare.
\bibitem{HM1}G. Hong, T. Ma, Vector-valued $q$-variation for differential operators and semigroups I, Math.Z \textbf{286} (2017), no. 1-2, 89-120.
\bibitem{HRW} G. Hong, S. Ray, S. Wang, Maximal ergodic inequalities for some positive operators on noncommutative $L_p$-spaces, to appear in J. London Math. Soc.
\bibitem{HS} G. Hong, M. Sun, Noncommutative multi-parameter Wiener-Wintner type ergodic theorem, J. Funct. Anal. \textbf{275} (2018), no. 5, 1100-1137.
\bibitem{HX}G. Hong, B. Xu, A noncommutative weak $(1,1)$ type estimate for a square function from ergodic theory, J. Funct. Anal. \textbf{280} (2021), no.9, 108959.
\bibitem{Hu} Y. Hu, Maximal ergodic theorems for some group actions, J. Funct. Anal. \textbf{254} (2008), no. 5, 1282-1306.
\bibitem{Ja1} R. Jajte, Strong limit theorems in noncommutative probability, Lecture Notes in Mathematics, volume 1110, Springer-Verlag, Berlin, 1985.
\bibitem{Ja2} R. Jajte, Strong limit theorems in noncommutative $L_2$-spaces, Lecture Notes in Mathematics, volume 1477, Springer-Verlag, Berlin, 1991.
\bibitem{JOR96}R. Jones, I. Ostrovskii, J. Rosenblatt, Square functions in ergodic theory, Ergod. Th. \& Dynam. Sys. \textbf{16} (1996), no. 2, 267-305.
\bibitem{JRW98}R. Jones, R. Kaufman, J. Rosenblatt, M. Wierdl, Oscillation in ergodic theory, Ergod. Th. \& Dynam. Sys. \textbf{18} (1998), no. 4, 889-935.
\bibitem{JRW03} R. Jones, J. Rosenblatt, and M. Wierdl, Oscillation in ergodic theory: higher dimensional results, Israel J. Math. 135 (2003), 1-27.
\bibitem{JMX}M. Junge, C. Le Merdy, Q. Xu, $H^{\infty}$ functional calculus and square functions on noncommutative $L_{p}$-spaces, Ast\'erisque. \textbf{305} (2006) vi+138 pp.
\bibitem{Jun-Ruan-Sher05}M. Junge, Z. Ruan, D. Sherman, A classification for 2-isometries of noncommutative $L_p$-spaces, Israel J. Math. \textbf{150} (2005) 285-314.
\bibitem{JX}M. Junge, Q. Xu, Noncommutative maximal ergodic theorems, J. Amer. Math. Soc. \textbf{20} (2007), no. 2, 385-439.
\bibitem{Ko-Sz}I. Kov\'{a}cs, J. Sz\"{u}cs, Ergodic type theorems in von Neumann algebras, Acta Sci. Math. (Szeged) \textbf{27} (1966), 233-246.
\bibitem{Kre} U. Krengel, On the speed of convergence in the ergodic theorem, Monatsh. Math. \textbf{86} (1978/79), no. 1, 3-6.
\bibitem{Kum78} B. Kummerer, A non-commutative individual ergodic theorem, Invent. Math., \textbf{46} (1978), 139-145.
\bibitem{Lan76}E. C. Lance, Ergodic theorems for convex sets and operator algebras,
Invent. Math. \textbf{37} (1976) 201--214.
\bibitem{Li}S. Litvinov, A non-commutative Wiener-Wintner theorem, Illinois J. Math. \textbf{58} (2014), no. 3, 697-708.
\bibitem{LP86} F. Lust-Piquard, In\'{e}galit\'{e}s de Khintchine dans $C_p~(1<p<\8)$, C. R. Acad. Sci. Paris S\'{e}r. I Math. \textbf{303} (1986), no. 7, 289-292.
\bibitem{LP-Pi91} F. Lust-Piquard, G. Pisier, Noncommutative Khintchine and Paley inequalities, Ark. Mat. \textbf{29} (1991), no. 2, 241-260.
\bibitem{M}T. Mei, Operator valued Hardy spaces, Mem. Amer. Math. Soc. \textbf{188} (2007), no. 881, vi+64 pp.
\bibitem{MP}T. Mei, J. Parcet, Pseudo-localization of singular integrals and noncommutative Littlewood-Paley inequalities, Int. Math. Res. Not. \textbf{8} (2009) 1433-1487.
\bibitem{Mu} M. Musat, Interpolation between non-commutative BMO and non-commutative $L_{p}$-spaces, J. Funct. Anal. \textbf{202} (2003), no. 1, 195-225.
\bibitem{JP1}J. Parcet, Pseudo-localization of singular integrals and noncommutative Calder\'on-Zygmund theory, J. Funct. Anal. \textbf{256} (2009), no. 2, 509-593.
\bibitem{Pis98}G. Pisier, Non-commutative vector valued $L_p$-spaces and completely $p$-summing maps, Ast\'{e}risque \textbf{247} (1998), vi+131.
\bibitem{PX} G. Pisier, Q. Xu, Non-commutative martingale inequalities,  Comm. Math. Phys. \textbf{189} (1997), no. 3, 667-698.
\bibitem{P2}G. Pisier, Q. Xu, Noncommutative $L^p$ spaces, Handbook of geometry of Banach spaces, Vol. 2, 1459-1517, North-Holland, 2003.
\bibitem{Ran1} N.  Randrianantoanina, Square function inequalities for non-commutative martingales,  Israel J. Math. \textbf{140} (2004), 333-365.
\bibitem{Ran2} N.  Randrianantoanina, A weak type inequality for non-commutative martingales and applications,   Proc. London Math. Soc. \textbf{91} (2005), no. 2, 509-542.
\bibitem{Rie38}F. Riesz, Some mean ergodic theorems, J. Lond. Math. Soc. \textbf{13} (1938), no. 4, 274-278.
\bibitem{Rie41}F. Riesz, Another proof of the mean ergodic theorem, Acta Univ. Szeged. Sect. Sci. Math. \textbf{10} (1941), 75-76.
 \bibitem{Su12}F. Sukochev, D. Zanin, Johnson-Schechtman inequalities in the free probability theory. J. Funct. Anal. \textbf{263} (2012), no. 10, 2921-2948.
\bibitem{St65} E. St{\o}rmer, On the Jordan structure of C$^{*}$-algebras, Tran. Amer. Math. Soc. \textbf{120} (1965) 438-447.
\bibitem{Sz-Fo70} B. Sz.-Nagy, C. Foias, Harmonic Analysis of Operators on Hilbert Space, North-Holland, 1970.
\bibitem{vonNeu32} J. von Neumann, Proof of the quasi-ergodic hypothesis, Proc. Natl. Acad. Sci. USA \textbf{18} (1932), 70-82.
\bibitem{Ye} F. J. Yeadon,  Ergodic theorems for semifinite von Neumann algebras. I, J. London Math. Soc.  \textbf{16} (1977) 326-332.
\bibitem{Yean80}F. J. Yeadon, Ergodic theorems for semifinite von Neumann algebras. II, Math. Proc. Cambridge Philos. Soc. \textbf{88} (1980), no. 1, 135-147.
\bibitem{Yea81} F. J. Yeadon, Isometric of noncommutative $L^p$-spaces, Math. Proc. Cambridge Philos. Soc \textbf{90} (1981), no. 1, 41-50.




\end{thebibliography}
\end{document}